\documentclass[reqno, 12pt]{amsart}
\usepackage{mathrsfs}
\usepackage{amscd}
\usepackage{amsmath}
\usepackage{latexsym}
\usepackage{amsfonts}
\usepackage{amssymb}
\usepackage{amsthm}
\usepackage{graphicx}
\usepackage{hyperref}
\usepackage{makecell}
\usepackage{array}
\usepackage{booktabs}
\usepackage{multirow}
\usepackage{color,xcolor}
\usepackage[ruled,vlined]{algorithm2e}
\usepackage{graphicx}

\newtheorem{theorem}{Theorem}[section]
\newtheorem{proposition}[theorem]{Proposition}
\newtheorem{corollary}[theorem]{Corollary}

\newtheorem{remark}[theorem]{Remark}

\newcommand\R{\mathbb{R}}
\newcommand\G{\mathbb{G}}

\newcommand\He{\mathbb{H}}

\definecolor{rot}{rgb}{1.000,0.000,0.000}
\definecolor{blue}{rgb}{0.000,0.000,1.000}

\numberwithin{equation}{section}

\title[Inverse moving point source problem]{An inverse moving point source problem in electromagnetics}

\author[M. Li et al.]{Minghui Li, Guanghui Hu and Yue Zhao}

\begin{document}

\begin{abstract}

This paper is concerned with an inverse moving point source problem in electromagnetics. The aim is to reconstruct the moving orbit from the tangential components of magnetic fields taken at a finite number of observation points.
The distance function between each observation point and the moving point source is computed by solving a nonlinear ordinary differential equation with an initial value. This ODE system only involves the measurement data from the tangential trace of the magnetic field at observation points.
As a consequence, the dynamical measurement data recorded at four non-coplanar points are sufficient to reconstruct the orbit function. A Lipschitz stability is established for the inverse problem, and
numerical experiments are reported to demonstrate the effectiveness of the proposed method. Numerical examples have shown that the reconstructed error  
depends linearly on the noise level and
 that the wave speed is a critical factor affecting the relative error.
\end{abstract}

\maketitle

\section{Introduction}

In this article, we consider the time-dependent Maxwell equations in a homogeneous background medium:
\begin{equation}\label{me}
\mu_0\partial_t \boldsymbol{H}(\boldsymbol{x}, t) + \nabla \times
\boldsymbol{E}(\boldsymbol{x},t) = 0,\quad
\varepsilon_0\partial_t \boldsymbol{E}(\boldsymbol{x},t) - \nabla \times
\boldsymbol{H}(\boldsymbol{x},t) = -\sigma \boldsymbol E + {\boldsymbol J}(\boldsymbol x, t),
\end{equation}
for $(\boldsymbol{x}, t)\in\mathbb R^3\times \mathbb R^+$, where
$\boldsymbol E$ and $\boldsymbol H$ are the electric and magnetic fields,
respectively. The source function ${\boldsymbol J}$ represents the electric
current density, $\varepsilon_0$ and $\mu_0$ the dielectric permittivity and the
magnetic permeability, respectively, and $\sigma$ the electric conductivity
which is assumed to vanish identically. Eliminating the electric field $\boldsymbol E$ from \eqref{me},
we obtain the governing equation for the magnetic field $\boldsymbol H$:
\begin{equation}\label{mf}
 \frac{1}{c^2}\partial^2_{t}\boldsymbol H(\boldsymbol{x},t) +
\nabla\times(\nabla\times\boldsymbol H
 (\boldsymbol{x},t))= \nabla\times{\boldsymbol J}(\boldsymbol x, t)=: \boldsymbol F(\boldsymbol x, t),\quad
\boldsymbol{x}\in \mathbb R^3,~ t>0,
\end{equation}
which is supplemented by the homogeneous conditions
\begin{align}\label{ic}
\boldsymbol H(\boldsymbol{x},t) = \partial_t \boldsymbol H(\boldsymbol{x},t) =
0, \quad \boldsymbol x \in \mathbb R^3,\; t\leq 0.
\end{align}
Here $c = \sqrt{\frac{1}{\varepsilon_0\mu_0}}$ is the wave speed of the electromagnetic wave propagating through the background medium.
From the first equation in \eqref{me} and the initial conditions, we obtain
\[
\partial_t[\nabla \cdot\boldsymbol H(\boldsymbol{x}, t)] = 0,\quad \nabla\cdot\boldsymbol H(\boldsymbol{x},0) =0,
 \] which lead to $\nabla\cdot\boldsymbol H(\boldsymbol{x}, t)=0$ for all $t\in [0, +\infty)$. Hence, the equation \eqref{mf} becomes
\begin{equation}\label{mf_1}
 \partial^2_{t}\boldsymbol H(\boldsymbol{x},t) -
\Delta\boldsymbol H
 (\boldsymbol{x},t)= \boldsymbol F(\boldsymbol x, t),\quad
\boldsymbol{x}\in \mathbb R^3,~ t>0.
\end{equation}
In  this paper, the electromagnetic wave is assumed to be excited by a moving point source. Specifically, the source function
$\boldsymbol F$ is supposed to be the form of
\begin{align}\label{st}
 \boldsymbol F(\boldsymbol x, t)= \delta(\boldsymbol x - \boldsymbol a(t))\,\boldsymbol f(t),
\end{align}
where $\delta$ is the Dirac function, $\boldsymbol f: \mathbb R_+\rightarrow \mathbb R^3$ is a $C^1$-smooth vector-valued temporal function, and
$\boldsymbol a:\mathbb R_+\rightarrow\mathbb R^3$ is the orbit function of a
moving point source with $\boldsymbol a(t)\in C^2[0, \infty)$. Driven by practical applications, we make the
following assumptions throughout the paper:
\begin{enumerate}
%\item The source radiates only over a finite time period $[0, T]$ for some
%$T > 0$, i.e., $\boldsymbol f(t) = 0$ for $t\geq T$ and $t\leq 0$.
%\vskip5pt
\item The source moves in a bounded domain $D$, i.e., $\{\boldsymbol a(t): t\in [0, +\infty)\} \subset D$.
\vskip8pt
\item The speed of the moving source is slower than the wave speed, that is, $|\boldsymbol a^\prime(t)|\leq c_0<c$ for some $c_0>0$ and for all $t>0$.
\vskip8pt
\item The acceleration of the moving source is bounded above, that is, $|\boldsymbol a^{\prime\prime}(t)|\leq a_0$ for some $a_0>0$ and for all $t>0$.
\end{enumerate}

%These assumptions imply that the location of the moving point source is contained in
%$D$ and its speed is less that the propagation of the electromagnetic field in the background media.
Let $\Omega$ be a smooth domain with the boundary $\Gamma:=
\partial \Omega$ such that $D\subset\subset\Omega$.
Denote by $\boldsymbol \nu$  the outward unit normal vector on $\Gamma$.
%To ensure that sufficient information can be collected on the boundary, we add an assumption on the source: the set of all observation points
%\[
%\mathcal{O} := \Gamma \setminus \{\boldsymbol{x}\in \Gamma \mid \text{there exist } t \in (0, T_0] \text{ and } k \in \mathbb{R} \text{ such that } f(t) = k \boldsymbol{\nu}(\boldsymbol{x})\}
%\]
%is not a subset of any plane in three-dimensional space.
In this work, we study the inverse moving source problem of determining the orbit $\boldsymbol a(t)$ from boundary measurements of the tangential trace of the magnetic field over a finite time interval on $\Gamma.$ Specifically, we propose a numerical method to solve the following inverse problem:
\vskip5pt
\noindent\textbf{IP}: Given $\boldsymbol f(t)$, reconstruct the orbit $\boldsymbol a(t)$ of the moving point source for $t\in[0, T_0]$ from the tangential components of the magnetic field, $\boldsymbol H(\boldsymbol x_j, t) \times \boldsymbol \nu(\boldsymbol x_j)$, with $\boldsymbol x_j\in \Gamma$, $j=1,2,3,4$ and $t\in [0, T]$. Here $\boldsymbol x_j$ are four appropriately chosen observation points and $T>T_0>0$ is sufficiently large.

\begin{figure}
\centering
\includegraphics[width=0.6\textwidth]{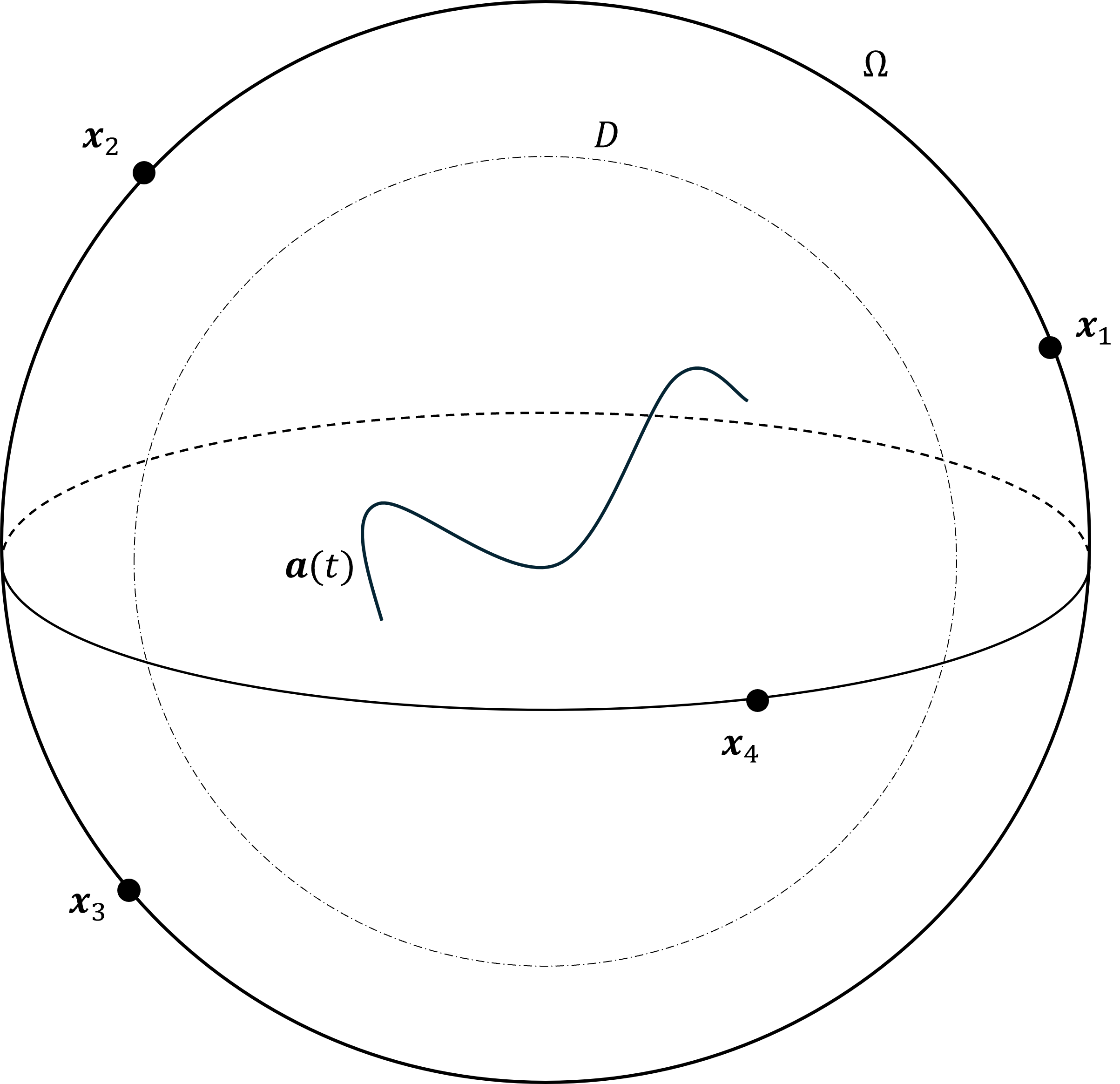}
\caption{Illustration of the observation points and the orbit.}
\label{figure:Illustration of the observation points and the orbit}
\end{figure}

Inverse moving source problems for wave equations aim at determining the unknown orbit of the moving source from the generated wave field measured
on the boundary. Such problems have important applications in scientific and industrial areas such as radar detection and biomedical imaging \cite{borcea_3}.
However, compared with the linear inverse source problems of determining the whole source term $\boldsymbol F(\boldsymbol x, t)$ investigated in
\cite{garnier_1, bao, hu, hu2, isakov, yama}, the inverse moving source problems are rarely studied. The reason is that the inverse moving source problem is highly nonlinear.
For acoustic wave equations with point moving source, numerical methods such as time-reversal,  Bayesian, algebraic and sampling methods can be found in \cite{garnier, alge, liu}.  In a recent work \cite{Hu_1}, a novel factorization method was proposed for recovering partial information on the trajectory of a moving point source.
A uniqueness result was proved in \cite{hu1} where the profile of the moving source is assumed to be a general $L^2$ function.
In the case of a moving acoustic point source, a Lipschitz stability estimate was derived in \cite{triki_ip}.
 For other works on the topic of detecting moving targets we refer the reader to \cite{borcea, borcea_1}.

This article is aimed to develop an efficient computational method for solving the inverse moving source problem in electromagnetic waves.
Specifically, fixing an observation point $\boldsymbol x$ on the boundary $\Gamma$, we introduce a distance function $v(t):=|\boldsymbol x - \boldsymbol a(t)|$,
which is the distance from the observation point to the point source at time $t$. Based on the representation of the magnetic field in terms of the Green tensor,
we derive an initial value nonlinear ordinary differential equation (ODE) satisfied by the distance function which involves the boundary measurements of the tangential magnetic field component $\boldsymbol H (\boldsymbol x)\times \boldsymbol \nu(\boldsymbol x) $ at $\boldsymbol x$. We compute the distance function by solving the ODE. Then by choosing four non-coplanar observation points on  $\Gamma$,
the location of the moving orbit can be determined by the corresponding four distance functions. Numerically, we locate the orbit by solving a system of linear equations.
Based on the computational method, we also derive a Lipschitz stability estimate for the inverse problem, which implies the uniqueness.
The idea of the proposed method traces back to \cite{hu2}, where a uniqueness result on the acoustic inverse moving source problem was proved. 
The stability proof in our work is inspired by an inverse moving point source problem in acoustics \cite{triki_ip}. However, in contrast with \cite{triki_ip}, our approach depends on the governing equation of the distant function proposed in  \cite{hu2},  which introduces a reformulated system with enhanced interpretability through physically meaningful variables.

This paper is organized as follows. In Section \ref{nm}, we propose a numerical method to solve the inverse moving source problem using the tangential trace of the magnetic field
 measured at four non-coplanar boundary points
over a finite period of time. In Section \ref{stabi}, we prove a Lipschitz stability for the inverse problem based on the proposed method. Numerical examples are presented to illustrate the effectiveness of the proposed method in Section \ref{numerics}.

\section{Numerical scheme}\label{nm}

In this section, we develop a numerical method of reconstructing the moving orbit $\boldsymbol a(t)$ from sparse boundary measurements. Let $\mathbb{G}$ be the Green tensor of the vector-valued wave equation which satisfies
\begin{equation}
\left\{
    \begin{array}{ll}
        \frac{1}{c^2}\partial_t^2\mathbb{G} - \Delta\mathbb{G} = \delta(t)\delta(\boldsymbol{x})\mathbb{I},   \\[5pt]
        \mathbb{G}(\boldsymbol{x},0) = \partial_t\mathbb{G}(\boldsymbol{x},0) = 0,
    \end{array}
\right. \notag
\end{equation}
where $\mathbb{I}\in\mathbb R^{3\times 3}$ is the identity matrix.
The Green tensor has the explicit form
\begin{equation}
    \mathbb{G}(\boldsymbol{x},t) = \frac{1}{4\pi |\boldsymbol{x}|}\delta\left(\frac{|\boldsymbol{x}|}{c}-t\right)\mathbb{I}, \notag
\end{equation}
Then the solution to the initial value problem \eqref{mf}--\eqref{ic} is given by
\begin{align}
    \boldsymbol{H}(\boldsymbol{x},t) &= \int_0^{+\infty} \int_{\mathbb{R}^3} \mathbb{G}(\boldsymbol{x-y},t-s)\boldsymbol{f}(s)\delta(\boldsymbol{y} - \boldsymbol{a}(s)) d\boldsymbol{y} ds \notag \\
    &= \int_0^{+\infty} \int_{\mathbb{R}^3} \frac{1}{4\pi |\boldsymbol{x-y}|}\delta\left(\frac{|\boldsymbol{x-y}|}{c}-t+s\right)\boldsymbol{f}(s)\delta(\boldsymbol{y} - \boldsymbol{a}(s)) d\boldsymbol{y} ds \notag \\
    &= \int_0^{+\infty} \frac{1}{4\pi |\boldsymbol{x}-\boldsymbol{a}(s)|} \delta\left(\frac{|\boldsymbol{x}-\boldsymbol{a}(s)|}{c}+s-t\right) \boldsymbol{f}(s) ds .\notag
\end{align}

For a given observation point $\boldsymbol{x}$ on $\Gamma$, define the function $$g(t):=t+\frac{|\boldsymbol{x}-\boldsymbol{a}(t)|}{c},\quad t\in(0,T_0).$$
A direct calculation gives
\begin{equation}
\frac{d}{dt}g(t)=1+\frac{\boldsymbol{a}'(t)\cdot(\boldsymbol{x}-\boldsymbol{a}(t))}{c|\boldsymbol{x}-\boldsymbol{a}(t)|}>1-\frac{c_0}{c}>0,\notag
\end{equation}
and thus $g(t)$ is a monotonically increasing continuous function. Therefore, there exists an inverse function of $g$ denoted by $g^{-1}(r)$, where $ r\in \left[\frac{|\boldsymbol{x}-\boldsymbol{a}(0)|}{c},T_0+\frac{|\boldsymbol{x}-\boldsymbol{a}(T_0)|}{c}\right]$.
Let $v(t; \boldsymbol{x}):=|\boldsymbol{x}-\boldsymbol{a}(t)|$ be the distance function between the observation point $\boldsymbol{x}$ and the point source at time $t$. Then
\begin{equation}
g(t)=t+\frac{v(t;\boldsymbol{x})}{c},\quad g'(t)=1+\frac{v'(t;\boldsymbol{x})}{c},
\end{equation}
where $v'$ denotes the derivative of $v$ with respect to the time variable $t$. Changing the variables, one can simplify the expression of $\boldsymbol{H}$ as follows: 
\begin{align}
    \boldsymbol{H}(\boldsymbol{x},t) &= \int_0^{+\infty} \frac{1}{4\pi |\boldsymbol{x}-\boldsymbol{a}(s)|} \delta\left(\frac{|\boldsymbol{x}-\boldsymbol{a}(s)|}{c}+s-t\right) \boldsymbol{f}(s) ds \notag \\
    &= \int_0^{+\infty} \frac{1}{4\pi |\boldsymbol{x}-\boldsymbol{a}(s)|} \delta(g(s)-t) \boldsymbol{f}(s) ds \notag \\
    &= \int_{g(0)}^{+\infty} \frac{1}{4\pi |\boldsymbol{x}-\boldsymbol{a}(g^{-1}(r))|} \delta(r-t) \boldsymbol{f}(g^{-1}(r)) \frac{1}{g^{'}(g^{-1}(r))} d\, r \notag \\
    &= \He(t-|\boldsymbol{x}-\boldsymbol{a}(0)|/c)\frac{1}{4\pi |\boldsymbol{x}-\boldsymbol{a}(g^{-1}(t))|} \boldsymbol{f}(g^{-1}(t)) \frac{1}{g^{'}(g^{-1}(t))} \notag\\
    &= \He(t-|\boldsymbol{x}-\boldsymbol{a}(0)|/c)\frac{1}{4\pi v(g^{-1}(t);\boldsymbol{x})} \boldsymbol{f}(g^{-1}(t)) \frac{c}{c+v^{'}(g^{-1}(t);\boldsymbol{x})},\label{Equality: Explicit expression for B  }
\end{align}
where $\He$ is the Heaviside function defined by
\begin{equation}
\He(t)=
\left\{
    \begin{array}{ll}
        0, & t < 0, \\[5pt]
        1, & t \geq 0.
    \end{array}
\right. \notag
\end{equation}

\begin{remark}
Since \( \boldsymbol{a} \) is a $C^2$-smooth function, \( v \) and \( g \) are $C^2$-smooth functions. Consequently, the inverse function \( g^{-1} \) is also $C^2$-smooth. Furthermore, since \( f \) is a $C^1$-smooth function, for any fixed \(\boldsymbol{x} \),  when $t>|\boldsymbol{x}-\boldsymbol{a}(0)|/c$, $\boldsymbol{H}$ is a $C^1$-smooth function with respect to time $t$.

\end{remark}

Replacing $t$ by $g(t)$ in (\ref{Equality: Explicit expression for B  }) gives
\begin{align}
    \boldsymbol{H}(\boldsymbol{x},g(t))
    =\frac{\He(g(t)-g(0))\,c}{4\pi v(t;\boldsymbol{x}) (c+v^{'}(t;\boldsymbol{x}))} \boldsymbol{f}(t)
    = \frac{c}{4\pi v(t;\boldsymbol{x}) (c+v^{'}(t;\boldsymbol{x}))} \boldsymbol{f}(t)
    \end{align}    
for all $t>0$, because
\[
g(t)-g(0)=t+|\boldsymbol{x}-\boldsymbol{a}(t)|/c-|\boldsymbol{x}-\boldsymbol{a}(0)|/c\geq0, \quad\mbox{if}\quad t>0.
\] 
Taking the cross product of both sides with the outward unit normal vector $\boldsymbol \nu$, we obtain for $t>0$ that
\begin{align}
&\boldsymbol{H}(\boldsymbol{x},t+\frac{v(t;\boldsymbol{x})}{c})\times \boldsymbol \nu(\boldsymbol{x})\\
=&\boldsymbol{H}(\boldsymbol{x},g(t))\times \boldsymbol \nu(\boldsymbol{x})\\
=&\frac{c}{4\pi v(t;\boldsymbol{x}) (c+v^{'}(t;\boldsymbol{x}))} \boldsymbol{f}(t)\times \boldsymbol \nu(\boldsymbol{x}).\label{eq:2}
\end{align}
Here we indicate the dependance on $v$ of the left and right hand sides of the above equation.
Denote by $H_i(\boldsymbol{x},t+v(t;\boldsymbol{x})/c)$ and $f_i(\boldsymbol{x},t)$ the $i$-th component of $\boldsymbol{H}(\boldsymbol{x},t+v(t;\boldsymbol{x})/c)\times \boldsymbol \nu(\boldsymbol{x})$ and $\boldsymbol{f}(t)\times \boldsymbol \nu(\boldsymbol{x})$, respectively. Then we have
%And for any \( t \in [0, T_0] \), there always exists \( i \in \{1,2,3\} \), such that the \( i \)-th component of \(\left( \boldsymbol{f}(t)\times \boldsymbol \nu(\boldsymbol{x})\right) \) is not 0, leading to the following equation
\begin{equation}
v^\prime(t;\boldsymbol{x})=\frac{cf_i(\boldsymbol{x},t)}{4\pi v(t;\boldsymbol{x}) H_i(\boldsymbol{x},t+v(t;\boldsymbol{x})/c)}-c, \,t\in(0,T_0),\label{vode}
\end{equation}
By our assumption on $\boldsymbol{a}$, it follows that \( v > 0 \), \( |v'| < c \) and that $v'$ is a $C^1$-smooth function with respect to $t>0$. This implies that the right hand side of \eqref{vode} is a well-defined function for all $t\in(0, T_0)$, even if the measurement data
$H_i(\boldsymbol{x},t+v(t;\boldsymbol{x})/c)$ vanishes at some moment. 

Denote the traveling time for the electromagnetic wave from the initial position $\boldsymbol{a}(0)$ of the point source to the observation point $\boldsymbol{x}$ by
\begin{equation}
    T(\boldsymbol x)=\inf\{t: \boldsymbol{H}(\boldsymbol{x},t)\neq \boldsymbol 0\}.\label{Equality: Definition of t_x  }
\end{equation}
According to the strong Huygens principle in three dimensions, we have
\begin{equation}
T(\boldsymbol x)=\frac{|\boldsymbol{x}-\boldsymbol{a}(0)|}{c}.\label{Equality: Explicit expression for t_x  }
\end{equation}
Therefore, the distance $|\boldsymbol{x} - \boldsymbol{a}(0)|$ from the initial position of the point source $\boldsymbol{a}(0)$ to the observation point $\boldsymbol{x}$
can be determined by the time when the receiver at $\boldsymbol{x}$ starts to receive the electromagnetic signal.
Moreover, combining the definition of $v(t;\boldsymbol{x})$ with equations (\ref{Equality: Definition of t_x  }) and (\ref{Equality: Explicit expression for t_x  }), we obtain the initial value of $v$ at $t=0$:
\begin{equation}
    v(0;\boldsymbol{x}) = c\, T(\boldsymbol x).\label{Equality: Initial value of v  }
\end{equation}

To reconstruct the moving orbit over the time interval $(0, T_0)$, we need the measurement data over a longer time period $(0, T)$ where $T>t_0$, due to the finite speed of the wave. 
we define the terminal time of the measurement data at $\boldsymbol x$ by
\begin{equation}
    T := T _0+ T_1, \notag
\end{equation}
where
%the electromagnetic wave to propagate from the point source to the boundary $\Gamma$,
%we introduce the additional measurement time $T_1$ as follows:
\begin{equation}\label{T1}
    T_1:=\sup\limits_{\boldsymbol{z}\in \Gamma,\boldsymbol{y}\in D}\left\{ \frac{|\boldsymbol{z}-\boldsymbol{y}|}{c} \right\}\geq\frac{|\boldsymbol{x}-\boldsymbol{a}(t)|}{c}\quad \mbox{for all}\quad \boldsymbol{x}\in\Gamma,\, t\in (0,T_0). 
\end{equation}
%Thus, to fully reconstruct the orbit of the moving source over the time period $[0, T_0]$,
%we consider the total observation time $T$ which is the sum of $T_0$ and $T_1$:
It satisfies the inequality
\begin{equation}
    T \geq t + \frac{|\boldsymbol{x}-\boldsymbol{a}(t)|}{c}=t + \frac{v(t;\boldsymbol{x})}{c}\quad\mbox{for all}\quad t\in (0,T_0),\label{Inequality: T_1 is an upper bound of g  }
\end{equation}
which guarantees that the data over a sufficiently large period of time are collected to reconstruct $\{a(t): t\in(0, T_0)\}$.

Now we describe a non-iterative reconstruction method based on the measurement data taken at four receivers $\boldsymbol{x}_j\in\Gamma$.
First, for any fixed $i\in\{1,2,3\}$
%given the tangential magnetic field component $\boldsymbol H \times \boldsymbol \nu$ measured at an observation point $\boldsymbol{x}\in\Gamma$ for $t\in[0,T]$,
%combing (\ref{Equality: Initial value of v  }), (\ref{Inequality: T_1 is an upper bound of g  }) and \eqref{vode},
we determine $v(t;\boldsymbol{x})$ by solving the initial value problem of the ordinary  differential equation
\begin{align}\label{vode_1}
\begin{cases}
        v'(t;\boldsymbol{x}) = \frac{cf_i(\boldsymbol{x},t)}{4\pi v(t;\boldsymbol{x}) H_i(\boldsymbol{x}, t + v(t;\boldsymbol{x})/c)}  - c, & t \in (0,T_0], \\[5pt]
        v(0;\boldsymbol{x}) = c \,T(\boldsymbol{x}).
 \end{cases}
\end{align}
It should be noticed that, in principle each index $i\in\{1,2,3\}$ can be used to determine $v$ and the right hand hand side is nonlinear with respect to $v$. In our numerical examples we are not restricted to a single index, because computational errors could appear if the denominator is much small. The system \eqref{vode_1} can be derived from the model $\eqref{mf_1}$ with the Dirac-type source term \eqref{st}, which however seems not applicable to extended moving sources. The Dirac distribution has been used to simplify the expression of the magnetic field \eqref{Equality: Explicit expression for B  }.
Define
$$
d_0:= \max_{\boldsymbol z\in\Gamma, \boldsymbol y\in D}|\boldsymbol z - \boldsymbol y|, \quad d_1:=\min_{\boldsymbol z\in\Gamma, \boldsymbol y\in D}|\boldsymbol z - \boldsymbol y|.
$$

%to solve this initial value problem. As \( t \) evolves, different \( i \) can be alternately used to avoid resulting from division by zero.

\begin{proposition}
Given $\boldsymbol x\in\Gamma,$
the initial value problem of the nonlinear ordinary differential equation \eqref{vode_1} admits a unique solution $v(\cdot, \boldsymbol x)\in C^1[0, T_0]$.
\end{proposition}

\begin{proof}
Without loss of generality, we consider the case $i=1$.
%we assume that \( f_1(\boldsymbol x,t) \neq 0 \) for all \( t \in [0, T_0] \), and consider the case where \( i = 1 \).
Define the domain
\[
U=\{(t, v): t\in [0, T_0], \, v\in (d_1, d_0)\cap (cT(\boldsymbol x)-ct, +\infty)\}.
\]
For $\boldsymbol x\in\Gamma$, we define the right hand of \eqref{vode_1} as  $ F(t, v) := \frac{cf_1(\boldsymbol{x},t)}{4\pi vH_1(\boldsymbol x, t+v/c)} - c.$
By the smoothness of $\boldsymbol{a}$, the function $H_1$ is of $C^1$-smoothness with respect to the second variable and $f_1$ is of $C^1$-smoothness with respect to the time variable. Therefore, $F$ is continuous on $U$ and $C^1$-smooth with respect to $v$. This implies that $F(t, v)$ satisfies the Lipschitz condition with respect to the variable $v$, proving the well-posedness of the initial value problem \eqref{vode_1} by Picard theorem. The proof is complete.
\end{proof}

Second, we shall determine $\boldsymbol{a}(t)$ from $v(t, \boldsymbol{x}_j)$ ($j=1,2,3,4$),
where $\boldsymbol{x}_1, \boldsymbol{x}_2, \boldsymbol{x}_3, \boldsymbol{x}_4$ are four observation points on $\Gamma$ which do not lie on one  plane.
%we show that if $v(t;\boldsymbol{x}_1), v(t;\boldsymbol{x}_2), v(t;\boldsymbol{x}_3), v(t;\boldsymbol{x}_4)$ are given, then $\boldsymbol{a}(t)$ can be uniquely determined by solving a system of four linear equations.
Denoting $\boldsymbol{a}(t)=\left(a_1(t), a_2(t), a_3(t)\right)$ and $\boldsymbol{x}_j=(x_j, y_j, z_j),\, j=1,2,3,4$. From the definition of $v(t;\boldsymbol{x})$, it follows that
\begin{equation}
\left\{
    \begin{array}{ll}
        v^2(t;\boldsymbol{x}_1) = \left( x_1 - a_1(t) \right)^2 + \left( y_1 - a_2(t) \right)^2 + \left( z_1 - a_3(t) \right)^2, \\[5pt]
        v^2(t;\boldsymbol{x}_2) = \left( x_2 - a_1(t) \right)^2 + \left( y_2 - a_2(t) \right)^2 + \left( z_2 - a_3(t) \right)^2, \\[5pt]
        v^2(t;\boldsymbol{x}_3) = \left( x_3 - a_1(t) \right)^2 + \left( y_3 - a_2(t) \right)^2 + \left( z_3 - a_3(t) \right)^2, \\[5pt]
        v^2(t;\boldsymbol{x}_4) = \left( x_4 - a_1(t) \right)^2 + \left( y_4 - a_2(t) \right)^2 + \left( z_4 - a_3(t) \right)^2.
    \end{array}
\right.
\end{equation}
Subtracting the above equations pairwisely and eliminating the quadratic terms, we obtain

\begin{equation}
\begin{bmatrix}
    v^2(t;\boldsymbol{x}_1) - v^2(t;\boldsymbol{x}_2) \\
    v^2(t;\boldsymbol{x}_2) - v^2(t;\boldsymbol{x}_3) \\
    v^2(t;\boldsymbol{x}_3) - v^2(t;\boldsymbol{x}_4)
\end{bmatrix} - X_1 = 2X_0
\begin{bmatrix}
    a_1(t) \\
    a_2(t) \\
    a_3(t)
\end{bmatrix},\notag
\end{equation}
where \begin{align*}
    X_1&=\begin{bmatrix}
    (x^2_1 - x^2_2) + (y^2_1 - y^2_2) + (z^2_1 - z^2_2) \\
    (x^2_2 - x^2_3) + (y^2_2 - y^2_3) + (z^2_2 - z^2_3) \\
    (x^2_3 - x^2_4) + (y^2_3 - y^2_4) + (z^2_3 - z^2_4)
\end{bmatrix}, \\ X_0&=\begin{bmatrix}
    x_1 - x_2 & y_1 - y_2 & z_1 - z_2 \\
    x_2 - x_3 & y_2 - y_3 & z_2 - z_3 \\
    x_3 - x_4 & y_3 - y_4 & z_3 - z_4
\end{bmatrix}.
\end{align*}
Since the four observation points are not coplanar, the determinant $|X_0| \neq 0$. Then we obtain the solution $\boldsymbol a(t)$ for every fixed $t$:
\begin{equation}
\begin{bmatrix}
    a_1(t) \\
    a_2(t) \\
    a_3(t)
\end{bmatrix} = \frac{1}{2}X_0^{-1}\begin{bmatrix}
    v^2(t;\boldsymbol{x}_1) - v^2(t;\boldsymbol{x}_2) \\
    v^2(t;\boldsymbol{x}_2) - v^2(t;\boldsymbol{x}_3) \\
    v^2(t;\boldsymbol{x}_3) - v^2(t;\boldsymbol{x}_4),
\end{bmatrix} - \frac{1}{2}X_0^{-1}X_1.\label{Equation: Linear system of equations satisfied by v  }
\end{equation}

% To ensure that four non-coplanar observation points can be found where the corresponding distance function \( v \) can be solved, we add an assumption on the source: the set of all observation points
% \[
% \mathcal{O} := \Gamma \setminus \{\boldsymbol{x}\in \Gamma \mid \text{there exist } t \in (0, T_0] \text{ and } k \in \mathbb{R} \text{ such that } f(t) = k \boldsymbol{\nu}(\boldsymbol{x})\}
% \]
% is not a subset of any plane in three-dimensional space.

We summarize the reconstruction method  as follows:
\begin{enumerate}
    \item Measure the tangential trace of the magnetic field $\boldsymbol{H}\times \boldsymbol \nu(\boldsymbol{x}, t)$ at four non-coplanar points $\boldsymbol{x}_1, \boldsymbol{x}_2, \boldsymbol{x}_3, \boldsymbol{x}_4$ %in the set $\mathcal{O}$
        over the time period $[0, T]$.\\
    \item Solve the ordinary differential equation (\ref{vode_1}) to get the distance function $v(t;\boldsymbol{x}_j)$ for $t \in (0,T_0), \, j=1,2,3,4$.\\
    \item Solve the linear system (\ref{Equation: Linear system of equations satisfied by v  }) to obtain the orbit function $\boldsymbol{a}(t)$ for $t \in (0,T_0)$.
\end{enumerate}

% ================================================

% \begin{remark}
% Since $v(t; \boldsymbol x)$ can be obtained by solving \eqref{vode_1} from $H_i$ for any $1\leq i\leq 3$,  it is possible to only utilize the tangential trace
% $\boldsymbol H\times\boldsymbol\nu$ on $\Gamma$. For example, let $\Omega = B_R=\{\boldsymbol x: |\boldsymbol x|\leq R\}$ and choose
% $\boldsymbol x_1 = (1, 0, 0), \boldsymbol x_2 = (0, 1, 0), \boldsymbol x_3 = (0, 0, 1), \boldsymbol x_4 = (0, 0, -1)$. Then we may measure
% $H_1$ at points $\boldsymbol x_2$, $\boldsymbol x_3$ and $\boldsymbol x_4$, and measure $H_2$ at $\boldsymbol x_1$, which can be achieved
% from $\boldsymbol H\times\boldsymbol\nu$ at those points. On the other hand, we can compute $\boldsymbol H\times\boldsymbol\nu_{\partial B_R}$ from
% $\boldsymbol E\times\boldsymbol\nu\vert_{\partial B_R}$ via the non-local Dirichlet-to-Neumann
% map $\boldsymbol H\times\boldsymbol\nu\vert_{\partial B_R} = T(\boldsymbol E\times\boldsymbol\nu\vert_{\partial B_R})$. In other words, to determine the orbit,
% we can also choose to measure the tangential trace of the electric field.

% \end{remark}

\section{Stability for the reconstruction method}\label{stabi}

This section concerns the stability estimate of the orbit function from four boundary
measurements.
For a column vector function $\boldsymbol{h}(t)=(h_1(t), h_2(t), h_3(t))^{\top}$, define the $L^\infty$- norm of $\boldsymbol{h}$ over the interval $(0,T_0)$ by
\begin{equation}
\left\|\boldsymbol{h}(\cdot)\right\|_{L^\infty(0,T_0)} := \underset{\substack{1\leq i\leq 3 \\ 0\leq t\leq T_0}}{\sup}\left |h_i(t)\right | .\notag
\end{equation}
The infinity norm of
 a column vector $\boldsymbol y=(y_1,y_2,y_3)^{\top}\in \R^3$ is defined as
\begin{equation}
\left\|\boldsymbol{y}\right\|_\infty := \underset{1\leq i\leq 3}{\max}|y_i| .\notag
\end{equation}
Then we define the infinity norm of a $3\times 3$ matrix $M$ by
\begin{equation}
\left\|M\right\|_\infty := \underset{|\boldsymbol y|\not = 0}{\max}\frac{\left \|M\boldsymbol y \right \|_\infty}{\left \| \boldsymbol y \right \|_\infty} .\notag
\end{equation}

Let $\boldsymbol H$ and $\widetilde{\boldsymbol H}$ be the magnetic fields generated by the moving orbits $\boldsymbol a$ and $\tilde{\boldsymbol  a}$, respectively.
The reconstruction method consists of three steps: calculate the arrival time of the signal, determine the distance
function and reconstruct the orbit. In what follows, motivated by \cite{triki_ip}, we estimate $T(\boldsymbol x) - \widetilde{T}(\boldsymbol x)$,
$v(t; \boldsymbol x) - \tilde{v}(t; \boldsymbol x)$ and $\boldsymbol a(t) - \tilde{\boldsymbol  a}(t)$ from a single boundary measurement data at $\boldsymbol x\in \Gamma$, under the assumption that 
\begin{equation}
\eta:=\underset{0\leq t\leq T}{\inf}\|\boldsymbol{f}(t)\times \boldsymbol{\nu}(\boldsymbol{x})\|_\infty>0.
\end{equation}

\newtheorem{theo1}{Theorem}
\begin{theo1}[Stability estimate of the signal arrival time] \label{thm1}
We have the stability estimate
\begin{equation}
|T(\boldsymbol x) - \widetilde{T}(\boldsymbol{x})| \leq C_0 \| \boldsymbol H(\boldsymbol{x}, \cdot)\times \boldsymbol \nu - \widetilde{\boldsymbol H}(\boldsymbol{x}, \cdot) \times \boldsymbol \nu \|_{L^{\infty}(0,T)},
\label{Inequality: Stability estimate for t_x  }
\end{equation}
where 
\begin{equation}\label{C0}
C_0=\frac{12\pi d_0(c+c_0)T}{c\,\eta} . 
\end{equation}
%$C_0$ is a constant dependent on $D,\,\Omega,\,c,\,c_0,\,\boldsymbol{f}$ and $T$.
\end{theo1}

\begin{proof} Without loss of generality, we assume $\widetilde{T}(\boldsymbol x) \leq T(\boldsymbol x)$ for some fixed $i=1,2,3$. Write $\widetilde{\boldsymbol H}\times \boldsymbol \nu =(\tilde{H}_1, \tilde{H}_2, \tilde{H}_3)$.
Then,
\begin{align*}
\int_{\widetilde{T}(\boldsymbol x)}^{T(\boldsymbol x)} |\tilde{H}_i(\boldsymbol{x}, s)| \, ds
&= \int_{\widetilde{T}(\boldsymbol x)}^{T(\boldsymbol x)} |\tilde{H}_i(\boldsymbol{x}, s) - H_i(\boldsymbol{x}, s)| \, ds \\
&\leq \int_{0}^{T} |\tilde{H}_i(\boldsymbol{x}, s) - H_i(\boldsymbol{x}, s)| \, ds \\
&\leq T \underset{0\leq t\leq T}{\sup}\left | \tilde{H}_i(\boldsymbol{x}, t) - H_i(\boldsymbol{x}, t) \right| \\
&\leq T \| \boldsymbol H(\boldsymbol x, \cdot)\times \boldsymbol \nu(\boldsymbol{x}) - \widetilde{\boldsymbol H}(\boldsymbol x,\cdot)\times \boldsymbol \nu (\boldsymbol{x}) \|_{L^{\infty}(0,T)} .
\end{align*}

Therefore,
\[
\int_{\widetilde{T}(\boldsymbol x)}^{T(\boldsymbol x)} \underset{1\leq i\leq 3}{\max}|\tilde{H}_i(\boldsymbol{x}, s)| \, ds\leq 3T \| \boldsymbol H(\boldsymbol x, \cdot)\times \boldsymbol \nu(\boldsymbol{x}) - \widetilde{\boldsymbol H}(\boldsymbol x, \cdot)\times \boldsymbol \nu (\boldsymbol{x}) \|_{L^{\infty}(0,T)} .
\]
Furthermore, for any $t >\widetilde{T}(\boldsymbol x)$, we have
\[
\widetilde{\boldsymbol H}(\boldsymbol{x},t)\times \boldsymbol\nu (\boldsymbol{x}) =\frac{1}{4\pi |\boldsymbol{x}-{\tilde{\boldsymbol a}}(\tilde{g}^{-1}(t))|} \boldsymbol{f}(\tilde{g}^{-1}(t))\times \boldsymbol\nu (\boldsymbol{x})\, \frac{1}{\tilde{g}^{'}(\tilde{g}^{-1}(t))}.
\]
Noting $0<\tilde{g}^{\prime}(s)=1+\frac{\tilde{\boldsymbol a}(s)\cdot(\boldsymbol{x}-\tilde{\boldsymbol a}(s))}{c|\boldsymbol{x}-\tilde{\boldsymbol a}(s)|}< 1+\frac{c_0}{c}$, we have
\begin{equation}
\underset{1\leq i\leq 3}{\max}|\tilde{H}_i(\boldsymbol{x},t)| \geq
\frac{c\underset{1\leq i\leq 3}{\max}|f_i(t)|}{4\pi d_0(c+c_0)} = \frac{c\|\boldsymbol{f}(t)\times \boldsymbol{\nu}(\boldsymbol{x})\|_\infty}{4\pi d_0(c+c_0)}, \notag
\end{equation}
where $d_0 := \underset{\boldsymbol{x}\in \Gamma,\boldsymbol{y}\in D}{\max}|\boldsymbol{x}-\boldsymbol{y}|$. Taking the infimum on both sides of the above inequality yields
\begin{equation}
\underset{0\leq t\leq T}{\inf}\underset{1\leq i\leq 3}{\max}|\tilde{H}_i(\boldsymbol{x},t)| \geq \frac{c\underset{0\leq t\leq T}{\inf}\|\boldsymbol{f}(t)\times \boldsymbol{\nu}(\boldsymbol{x})\|_\infty}{4\pi d_0(c+c_0)}. \notag
\end{equation}
Hence,
\begin{align*}
&\frac{c\underset{0\leq t\leq T}{\inf}\|\boldsymbol{f}(t)\times \boldsymbol{\nu}(\boldsymbol{x})\|_\infty}{4\pi d_0(c+c_0)}|T(\boldsymbol x) - \widetilde{T}(\boldsymbol x)| \\
&\leq
|T(\boldsymbol x) - \widetilde{T}(\boldsymbol x)|\underset{0\leq t\leq T}{\inf}\underset{1\leq i\leq 3}{\max}|\tilde{H}_i(\boldsymbol{x},t)|\\
&\leq \int_{\widetilde{T}(\boldsymbol x)}^{T(\boldsymbol x)} \underset{1\leq i\leq 3}{\max}|\tilde{H}_i(\boldsymbol{x}, s)| \, ds\\
&\leq 3T\| \boldsymbol H(\boldsymbol{x}, \cdot)\times \boldsymbol \nu - \widetilde{\boldsymbol H}(\boldsymbol{x}, \cdot)\times \boldsymbol\nu \|_{L^{\infty}(0,T)}.
% \underset{0\leq t\leq T_0}{\inf}\underset{1\leq i\leq 3}{\max}|\tilde{H}_i(\boldsymbol{x},t)| \geq \frac{c\underset{0\leq t\leq T_0}{\inf}\|\boldsymbol{f}(t)\times \boldsymbol{\nu}(\boldsymbol{x})\|_\infty}{4\pi d_0(c+c_0)}, \notag
\end{align*}
Finally, we obtain
\begin{equation}\label{eqn}
|T(\boldsymbol x) - \widetilde{T}(\boldsymbol x)| \leq C_0 \underset{0\leq t\leq T}{\sup}\left | \tilde{H}_i(\boldsymbol{x}, t) - H_i(\boldsymbol{x}, t) \right| ,
\end{equation}
where $C_0>0$ is given by \eqref{C0}.
%\begin{equation}
%C_0=\frac{12\pi d_0(c+c_0)T}{c\underset{0\leq t\leq T}{\inf}\|\boldsymbol{f}(t)\times \boldsymbol{\nu}(\boldsymbol{x})\|_\infty} . \notag
%\end{equation}
% Noting the stability \eqref{eqn} holds for any $i=1, 2, 3$, we complete the proof.
\end{proof}

\newtheorem{theo2}[theo1]{Theorem}
\begin{theo2}[Stability estimate of the distance function]\label{thm2}
For $i = 1, 2, 3$, there holds the stability
\begin{equation}
\underset{0\leq t\leq T_0}{\sup}|v(t;\boldsymbol{x}) - \tilde{v}(t;\boldsymbol{x})|\leq C_1\| H_i(\boldsymbol{x}, \cdot) - \tilde{H}_i(\boldsymbol{x}, \cdot) \|_{L^{\infty}(0,T)},
\label{Inequality: Stability estimate for v^2  }
\end{equation}
where $C_1$ is a constant dependent on $\Omega,\,c,\,c_0,a_0,\,\boldsymbol{f},\, T$.
\end{theo2}

\begin{proof} %Write $f=f_i$ and $H=H_i$ for fixed $i\in\{1,2,3\}$. 
Define the function \( f(\boldsymbol{x}, t) \) on the time interval \([0, T]\) as follows:
\[
f(\boldsymbol{x}, t) =
\begin{cases}
f_1(\boldsymbol{x}, t), & \text{if } |f_1(\boldsymbol{x}, t)| \geq |f_2(\boldsymbol{x}, t)| \text{ and } |f_1(\boldsymbol{x}, t)| \geq |f_3(\boldsymbol{x}, t)|, \\[5pt]
f_2(\boldsymbol{x}, t), & \text{if } |f_2(\boldsymbol{x}, t)| > |f_1(\boldsymbol{x}, t)| \text{ and } |f_2(\boldsymbol{x}, t)| \geq |f_3(\boldsymbol{x}, t)|, \\[5pt]
f_3(\boldsymbol{x}, t), & \text{if otherwise},
\end{cases}
\]
where $f_i$ ($i=1,2,3$) denotes the $i$-th component of the function $\boldsymbol{f}$.
Analogously, define \( H(\boldsymbol{x}, t) \) on the time interval \([0, T]\) by
\[
H(\boldsymbol{x}, t) =
\begin{cases}
H_1(\boldsymbol{x}, t), & \text{if } |H_1(\boldsymbol{x}, t)| \geq |H_2(\boldsymbol{x}, t)| \text{ and } |H_1(\boldsymbol{x}, t)| \geq |H_3(\boldsymbol{x}, t)|, \\[5pt]
H_2(\boldsymbol{x}, t), & \text{if } |H_2(\boldsymbol{x}, t)| > |H_1(\boldsymbol{x}, t)| \text{ and } |H_2(\boldsymbol{x}, t)| \geq |H_3(\boldsymbol{x}, t)|, \\[5pt]
H_3(\boldsymbol{x}, t), & \text{if otherwise}.
\end{cases}
\]
By \eqref{eq:2}, the component indexes corresponding to $H$ and $f$ are consistent for all $t\in(0, T_0)$, i.e., the relation $f=f_j$ for some $j\in{1,2,3}$ and $t>0$ implies that $H=H_j$ with the same sub-index $j$.
Let $\tilde{H}(\boldsymbol{x}, t) $ and $\tilde{v}(t; \boldsymbol{x})$ be defined as the same as $H(\boldsymbol{x}, t)$ and $v(t; \boldsymbol{x})$. 
%And also define \( \tilde{H}(\boldsymbol{x}, t) \) on the time interval \([0, T]\) as follows:
%\[
%\tilde{H}(\boldsymbol{x}, t) =
%\begin{cases}
%\tilde{H}_1(\boldsymbol{x}, t), & \text{if } |\tilde{H}_1(\boldsymbol{x}, t)| \geq |\tilde{H}_2(\boldsymbol{x}, t)| \text{ and } |\tilde{H}_1(\boldsymbol{x}, t)| \geq |\tilde{H}_3(\boldsymbol{x}, t)|, \\[5pt]
%\tilde{H}_2(\boldsymbol{x}, t), & \text{if } |\tilde{H}_2(\boldsymbol{x}, t)| > |\tilde{H}_1(\boldsymbol{x}, t)| \text{ and } |\tilde{H}_2(\boldsymbol{x}, t)| \geq |\tilde{H}_3(\boldsymbol{x}, t)|, \\[5pt]
%\tilde{H}_3(\boldsymbol{x}, t), & \text{otherwise}.
%\end{cases}
%\]
According to \eqref{vode}, we have
$$
2v'(t;\boldsymbol{x})v(t;\boldsymbol{x}) = \frac{cf(\boldsymbol{x},t)}{2\pi  H(\boldsymbol{x}, t + v(t;\boldsymbol{x})/c)} - 2cv(t;\boldsymbol{x}), \quad t \in (0,T_0],
$$
and
$$
2\tilde v'(t;\boldsymbol{x})\tilde v(t;\boldsymbol{x}) = \frac{cf(\boldsymbol{x},t)}{2\pi  \tilde{H}(\boldsymbol{x}, t + \tilde v(t;\boldsymbol{x})/c)} - 2c\tilde v(t;\boldsymbol{x}), \quad t \in (0,T_0].
$$
Integrating these two equations from 0 to \( t\in(0,T_0] \) and then taking their difference, we obtain
\begin{align*}
v^2(t;\boldsymbol{x}) - \tilde{v}^2(t;\boldsymbol{x}) &=
c^2(T(\boldsymbol x)^2-\widetilde{T}(\boldsymbol x)^2)  \notag \\
&\quad+ 2c\int_0^t \left[v(s;\boldsymbol{x}) - \tilde{v}(s;\boldsymbol{x})\right] dt\notag\\
&\quad+ \frac{c}{2\pi}\int_0^t \left[\frac{f(\boldsymbol{x},s)}{H(\boldsymbol{x},s+v(s;\boldsymbol{x})/c)}-\frac{f(\boldsymbol{x},s)}{\tilde{H}(\boldsymbol{x},s+\tilde{v}(s;\boldsymbol{x})/c)}\right]ds .
\end{align*}
We estimate each term of the above expression as follows. Introduce $\xi(t) = |v(t;\boldsymbol{x}) - \tilde{v}(t;\boldsymbol{x})|$.
\begin{enumerate}
    \item From the definitions of $d_1$, we know that $$|v^2(t;\boldsymbol{x}) - \tilde{v}^2(t;\boldsymbol{x})| = |v(t;\boldsymbol{x}) - \tilde{v}(t;\boldsymbol{x})||v(t;\boldsymbol{x}) + \tilde{v}(t;\boldsymbol{x})|\geq 2d_1\xi(t).$$
\item From the definitions of $T_1$ and $d_0$, we know that $\max\{T(\boldsymbol x),\widetilde{T}(\boldsymbol x)\}\leq T_1=d_0/c$. Then in view of (\ref{Inequality: Stability estimate for t_x  }),

\begin{align*}
|T(\boldsymbol x)^2-\widetilde{T}(\boldsymbol x)^2|&=|(T(\boldsymbol x)+\widetilde{T}(\boldsymbol x))(T(\boldsymbol x)-\widetilde{T}(\boldsymbol x))|  \notag \\
&\leq \frac{2d_0}{c} C_0 \| \boldsymbol H(\boldsymbol{x}, \cdot)\times \boldsymbol \nu - \widetilde{\boldsymbol H}(\boldsymbol{x}, \cdot)\times \boldsymbol \nu \|_{L^{\infty}(0,T)}.
\end{align*}
    \item By the definition of $\xi(t)$, we know
    $$|\int_0^t \left[v(s;\boldsymbol{x}) - \tilde{v}(s;\boldsymbol{x})\right] dt|\leq \int_0^t \xi(t)dt.$$
    \item By the triangle inequality,
   \begin{align}\nonumber
    &\quad \left|\int_0^t \frac{f(\boldsymbol{x},s)}{H(\boldsymbol{x},s+v(s;\boldsymbol{x})/c)} -\frac{f(\boldsymbol{x},s)}{\tilde{H}(\boldsymbol{x},s+\tilde{v}(s;\boldsymbol{x})/c)} ds \right| \\ \nonumber
    &\leq \int_0^t \left| \frac{f(\boldsymbol{x},s)}{H(\boldsymbol{x},s+v(s;\boldsymbol{x})/c)}-\frac{f(\boldsymbol{x},s)}{\tilde{H}(\boldsymbol{x},s+\tilde{v}(s;\boldsymbol{x})/c)}\right|ds \\ \nonumber
    &\leq \int_0^t |f(\boldsymbol{x},s)|\left| \frac{\tilde{H}(\boldsymbol{x},s+\tilde{v}(s;\boldsymbol{x})/c)-H(\boldsymbol{x},s+\tilde{v}(s;\boldsymbol{x})/c)}{H(\boldsymbol{x},s+v(s;\boldsymbol{x})/c)\tilde{H}(\boldsymbol{x},s+\tilde{v}(s;\boldsymbol{x})/c)} \right|ds \\ \label{eq:1}
    &\quad + \int_0^t |f(\boldsymbol{x},s)|\left| \frac{H(\boldsymbol{x},s+\tilde{v}(s;\boldsymbol{x})/c)-H(\boldsymbol{x},s+v(s;\boldsymbol{x})/c)}{H(\boldsymbol{x},s+v(s;\boldsymbol{x})/c)\tilde{H}(\boldsymbol{x},s+\tilde{v}(s;\boldsymbol{x})/c)} \right|ds.
    \end{align}
    According to (\ref{Equality: Explicit expression for B  }), for any $t>0$, we have
    \begin{equation*}
    |H(\boldsymbol{x},t+v(t;\boldsymbol{x})/c)|=|H(\boldsymbol{x},g(t))| = \frac{1}{4\pi v(t,\boldsymbol{x})} \frac{|f(\boldsymbol{x}, t)|}{g^{'}(t)} \geq \frac{c|f(\boldsymbol{x}, t)|}{4\pi d_0(c+c_0)}.
    \end{equation*}
    Therefore, the first integral term on the right hand side of \eqref{eq:1} can be bounded by
    \begin{align*}
&\quad \int_0^t |f(\boldsymbol{x},s)|\left| \frac{\tilde{H}(\boldsymbol{x},s+\tilde{v}(s;\boldsymbol{x})/c)-H(\boldsymbol{x},s+\tilde{v}(s;\boldsymbol{x})/c)}{H(\boldsymbol{x},s+v(s;\boldsymbol{x})/c)\tilde{H}(\boldsymbol{x},s+\tilde{v}(s;\boldsymbol{x})/c)} \right|ds \notag \\
&\leq \int_0^t \frac{16\pi^2d_0^2(c+c_0)^2}{c^2} \frac{|\tilde{H}(\boldsymbol{x},s+\tilde{v}(s;\boldsymbol{x})/c)-H(\boldsymbol{x},s+v(s;\boldsymbol{x})/c)|}{|f(\boldsymbol{x},s)|}ds \notag \\
&\leq \frac{16\pi^2d_0^2(c+c_0)^2T}{c^2\underset{0\leq t\leq T}{\inf}\|\boldsymbol{f}(\boldsymbol{x},t)\|_\infty} \underset{0<t\leq T}{\sup}| H(\boldsymbol{x}, t) - \tilde{H}(\boldsymbol{x}, t) |\\
&= \frac{16\pi^2d_0^2(c+c_0)^2T}{c^2\underset{0\leq t\leq T}{\inf}\|\boldsymbol{f}(t)\times \boldsymbol{\nu}(\boldsymbol{x})\|_\infty}\| \boldsymbol{H}(\boldsymbol{x}, \cdot)\times\boldsymbol{\nu} - \widetilde{\boldsymbol H}(\boldsymbol{x}, \cdot)\times \boldsymbol \nu \|_{L^{\infty}(0,T)}.
\end{align*}
Then, by taking the cross product of both sides of Equation \eqref{Equality: Explicit expression for B  } with the outward normal vector \(\boldsymbol{\nu}\) and subsequently differentiating with respect to $t$, we obtain the following upper bound estimate:
\begin{align*}
    \frac{4\pi}{c}\|(\boldsymbol{H}\times\boldsymbol{\nu})'(\boldsymbol{x},t)\|_\infty
    &\leq \left\|\frac{c(\boldsymbol{f}\times\boldsymbol{\nu})'(\boldsymbol{x},g^{-1}(t))}{v(g^{-1}(t);\boldsymbol{x})\left(c+v^{'}(g^{-1}(t);\boldsymbol{x})\right)^2}\right\|_\infty\\
    &\quad + \left\|\frac{cv'(g^{-1}(t);\boldsymbol{x})(\boldsymbol{f}\times\boldsymbol{\nu})(\boldsymbol{x},g^{-1}(t))}{v^2(g^{-1}(t);\boldsymbol{x})\left(c+v^{'}(g^{-1}(t);\boldsymbol{x})\right)^2}\right\|_\infty\\
    &\quad + \left\|\frac{cv{''}(g^{-1}(t);\boldsymbol{x})(\boldsymbol{f}\times\boldsymbol{\nu})(\boldsymbol{x},g^{-1}(t))}{v(g^{-1}(t);\boldsymbol{x})\left(c+v^{'}(g^{-1}(t);\boldsymbol{x})\right)^3}\right\|_\infty\\
    &\leq \frac{c}{d_1(c-c_0)^2}\underset{0\leq t\leq T}{\sup}\|(\boldsymbol{f}\times\boldsymbol{\nu})'(\boldsymbol{x},t)\|_\infty\\
    &\quad +\frac{cc_0}{d_1^2(c-c_0)^2}\underset{0\leq t\leq T}{\sup}\|(\boldsymbol{f}\times\boldsymbol{\nu})(\boldsymbol{x},t)\|_\infty\\
    &\quad + \frac{ca_0}{d_1(c-c_0)^3}\underset{0\leq t\leq T}{\sup}\|(\boldsymbol{f}\times\boldsymbol{\nu})(\boldsymbol{x},t)\|_\infty.
\end{align*}
Then we can obtain an upper bound estimate for the second integral term :
\begin{align*}
    &\quad \int_0^t |f(\boldsymbol{x},s)|\left| \frac{H(\boldsymbol{x},s+\tilde{v}(s;\boldsymbol{x})/c)-H(\boldsymbol{x},s+v(s;\boldsymbol{x})/c)}{H(\boldsymbol{x},s+v(s;\boldsymbol{x})/c)\tilde{H}(\boldsymbol{x},s+\tilde{v}(s;\boldsymbol{x})/c)} \right|ds \\
    &\leq \int_0^t \frac{16\pi^2d_0^2(c+c_0)^2}{c^2} \frac{| H(\boldsymbol{x},s+\tilde{v}(s;\boldsymbol{x})/c)-H(\boldsymbol{x},s+v(s;\boldsymbol{x})/c)|}{|f(\boldsymbol{x},s)|}ds \\
    &\leq \frac{2\underset{0\leq t\leq T}{\sup}\|(\boldsymbol{H}\times\boldsymbol{\nu})'(\boldsymbol{x},t)\|_\infty}
    {\underset{0\leq t\leq T}{\inf}\|(\boldsymbol{f}\times\boldsymbol{\nu})(\boldsymbol{x},t)\|_\infty}
    \frac{16\pi^2d_0^2(c+c_0)^2}{c^3} \int_0^t |\tilde{v}(s;\boldsymbol{x})-v(s;\boldsymbol{x})|ds\\
    &\leq \left(\frac{8\pi d_0^2(c+c_0)^2\underset{0\leq t\leq T}{\sup}\|(\boldsymbol{f}\times\boldsymbol{\nu})(\boldsymbol{x},t)\|_\infty}
    {d_1c(c-c_0)^2\underset{0\leq t\leq T}{\inf}\|(\boldsymbol{f}\times\boldsymbol{\nu})(\boldsymbol{x},t)\|_\infty}\right)\times \\
    & \quad\left( \frac{\underset{0\leq t\leq T}{\sup}\|(\boldsymbol{f}\times\boldsymbol{\nu})'(\boldsymbol{x},t)\|_\infty}{\underset{0\leq t\leq T}{\sup}\|(\boldsymbol{f}\times\boldsymbol{\nu})(\boldsymbol{x},t)\|_\infty} + (\frac{c_0}{d_1}+\frac{a_0}{c-c_0})\right) \int_0^t\xi(s)ds.
    \end{align*}
    Thus, there exist constants \( A_1 \) and \( A_2 \) such that
    \begin{align*}
    &\quad \left|\int_0^t \frac{f(\boldsymbol{x},s)}{H(\boldsymbol{x},s+v(s;\boldsymbol{x})/c)} -\frac{f(\boldsymbol{x},s)}{\tilde{H}(\boldsymbol{x},s+\tilde{v}(s;\boldsymbol{x})/c)} ds \right| \\
    &\quad \leq A_1\| \boldsymbol{H}(\boldsymbol{x}, \cdot)\times\boldsymbol{\nu} - \widetilde{\boldsymbol H}(\boldsymbol{x}, \cdot)\times \boldsymbol \nu\|_{L^{\infty}(0,T)} + A_2\int_0^t\xi(s)ds.
    \end{align*}
\end{enumerate}
Combining the above discussions, we can obtain the inequality
$$
\xi(t)\leq B_1\| \boldsymbol{H}(\boldsymbol{x}, \cdot)\times\boldsymbol{\nu} - \widetilde{\boldsymbol H}(\boldsymbol{x}, \cdot)\times \boldsymbol \nu\|_{L^{\infty}(0,T)}+B_2\int_0^t\xi(s)ds.
$$
Here, \( B_1 \) and \( B_2 \) are constants depending on \( \Omega \), \( c \),\( c_0 \),\( a_0 \), \( \boldsymbol{f} \) and \( T \). By applying the Gronwall's inequality we obtain
\begin{align*}
\xi(t)\leq \left( 1+B_2T  e^{B_2T}\right)  B_1\| \boldsymbol{H}(\boldsymbol{x}, \cdot)\times\boldsymbol{\nu} - \widetilde{\boldsymbol H}(\boldsymbol{x}, \cdot)\times \boldsymbol \nu\|_{L^{\infty}(0,T)},
\end{align*}which leads to the error estimate
\[
    |v(t;\boldsymbol{x}) - \tilde{v}(t;\boldsymbol{x})|  \\
    \leq C_1 \| \boldsymbol{H}(\boldsymbol{x}, \cdot)\times\boldsymbol{\nu} - \widetilde{\boldsymbol H}(\boldsymbol{x}, \cdot)\times \boldsymbol \nu\|_{L^{\infty}(0,T)} ,
\]
where \( C_1 \) depends on \( \Omega \), \( c \), \( c_0 \), \( a_0 \), \( \boldsymbol{f} \) and \( T \).

In addition, we also have the upper bound
\begin{align*}
    |v^2(t;\boldsymbol{x}) - \tilde{v}^2(t;\boldsymbol{x})| &\leq 2d_0 |v(t;\boldsymbol{x}) - \tilde{v}(t;\boldsymbol{x})| \\
    &\leq 2d_0 C_1\, \| \boldsymbol{H}(\boldsymbol{x}, \cdot)\times\boldsymbol{\nu} - \widetilde{\boldsymbol H}(\boldsymbol{x}, \cdot)\times \boldsymbol \nu\|_{L^{\infty}(0,T)} .
\end{align*}
\end{proof}

With the help of Theorems \ref{thm1} and \ref{thm2}, we obtain the stability for the reconstruction of the orbit function, which also implies uniqueness
of the inverse problem.
\newtheorem{theo4}[theo1]{Theorem}
\begin{theo4}[Stability estimate for the orbit function \texorpdfstring{$\boldsymbol{a}$}{b}]\label{thm}
Let $\boldsymbol{x_1},\,\boldsymbol{x_2},\,\boldsymbol{x_3},\,\boldsymbol{x_4}$ be four non-coplanar points on $\Gamma$.
It holds that
\begin{equation}\label{stability}
\left\|\boldsymbol{a}(\cdot)-\tilde{\boldsymbol{a}}(\cdot) \right\|_{L^\infty(0,T_0)}\leq C_2 \| \boldsymbol{H}(\boldsymbol{x}, \cdot)\times\boldsymbol{\nu} - \widetilde{\boldsymbol H}(\boldsymbol{x}, \cdot)\times \boldsymbol \nu\|_{L^{\infty}(0,T)},
\end{equation}
where $C_2$ is a constant dependent on $\Omega,\,c,\,c_0, a_0,\,\boldsymbol{f},\,T$ and $\boldsymbol{x_j}$,$j=1,2,3,4$.
\end{theo4}

\begin{proof}

From (\ref{Equation: Linear system of equations satisfied by v  }), we have that for any fixed $t\in(0,T_0)$,
\begin{align}
\begin{bmatrix}
    a_1(t)-\tilde{a}_1(t) \\
    a_2(t)-\tilde{a}_2(t) \\
    a_3(t)-\tilde{a}_3(t)
\end{bmatrix}&=\frac{1}{2}X_0^{-1}\left(\begin{bmatrix}
    v^2(t;\boldsymbol{x}_1)-v^2(t;\boldsymbol{x}_2)  \\
    v^2(t;\boldsymbol{x}_2)-v^2(t;\boldsymbol{x}_3)  \\
    v^2(t;\boldsymbol{x}_3)-v^2(t;\boldsymbol{x}_4)
\end{bmatrix}-\begin{bmatrix}
    \tilde{v}^2(t;\boldsymbol{x}_1)-\tilde{v}^2(t;\boldsymbol{x}_2)  \\
    \tilde{v}^2(t;\boldsymbol{x}_2)-\tilde{v}^2(t;\boldsymbol{x}_3)  \\
    \tilde{v}^2(t;\boldsymbol{x}_3)-\tilde{v}^2(t;\boldsymbol{x}_4)
\end{bmatrix}\right) \notag \\
&=\frac{1}{2}X_0^{-1}\left(\begin{bmatrix}
    v^2(t;\boldsymbol{x}_1)-\tilde{v}^2(t;\boldsymbol{x}_1)  \\
    v^2(t;\boldsymbol{x}_2)-\tilde{v}^2(t;\boldsymbol{x}_2)  \\
    v^2(t;\boldsymbol{x}_3)-\tilde{v}^2(t;\boldsymbol{x}_3)
\end{bmatrix}-\begin{bmatrix}
    v^2(t;\boldsymbol{x}_2)-\tilde{v}^2(t; \boldsymbol{x}_2)  \\
    v^2(t;\boldsymbol{x}_3)-\tilde{v}^2(t; \boldsymbol{x}_3)  \\
    v^2(t;\boldsymbol{x}_4)-\tilde{v}^2(t;\boldsymbol{x}_4)
\end{bmatrix}\right) . \notag
\end{align}
Using the triangle inequality and the definition of matrix norms, we get
\begin{align}
\left\|\begin{bmatrix}
    a_1(t)-\tilde{a}_1(t) \\
    a_2(t)-\tilde{a}_2(t) \\
    a_3(t)-\tilde{a}_3(t)
\end{bmatrix}\right\|_\infty\leq&\frac{1}{2}\left\|X_0^{-1}\right\|_\infty \left\|\begin{bmatrix}
    v^2(t;\boldsymbol{x}_1)-\tilde{v}^2(t;\boldsymbol{x}_1)  \\
    v^2(t;\boldsymbol{x}_2)-\tilde{v}^2(t;\boldsymbol{x}_2)  \\
    v^2(t;\boldsymbol{x}_3)-\tilde{v}^2(t;\boldsymbol{x}_3)
\end{bmatrix}\right\|_\infty \notag\\
&\quad+\frac{1}{2}\left\|X_0^{-1}\right\|_\infty\left\|\begin{bmatrix}
    v^2(t;\boldsymbol{x}_2)-\tilde{v}^2(t;\boldsymbol{x}_2)  \\
    v^2(t;\boldsymbol{x}_3)-\tilde{v}^2(t;\boldsymbol{x}_3)  \\
    v^2(t\boldsymbol{x}_4)-\tilde{v}^2(t;\boldsymbol{x}_4)
\end{bmatrix}\right\|_\infty .\notag
\end{align}
By (\ref{Inequality: Stability estimate for v^2  }), the right-hand side of the previous inequality can be bounded by
\begin{align}&
\frac{1}{2}\left\|X_0^{-1}\right\|_\infty \left(\left\|\begin{bmatrix}
    v^2(t;\boldsymbol{x}_1)-\tilde{v}^2(t;\boldsymbol{x}_1)  \\
    v^2(t;\boldsymbol{x}_2)-\tilde{v}^2(t;\boldsymbol{x}_2)  \\
    v^2(t;\boldsymbol{x}_3)-\tilde{v}^2(t;\boldsymbol{x}_3)
\end{bmatrix}\right\|_\infty+\left\|\begin{bmatrix}
    v^2(t;\boldsymbol{x}_2)-\tilde{v}^2(t;\boldsymbol{x}_2)  \\
    v^2(t;\boldsymbol{x}_3)-\tilde{v}^2(t;\boldsymbol{x}_3)  \\
    v^2(t;\boldsymbol{x}_4)-\tilde{v}^2(t;\boldsymbol{x}_4)
\end{bmatrix}\right\|_\infty\right) \notag\\
&\quad\quad\leq 2d_0\left\|X_0^{-1}\right\|_\infty
C_1 \| \boldsymbol{H}(\boldsymbol{x}, \cdot)\times\boldsymbol{\nu} - \widetilde{\boldsymbol H}(\boldsymbol{x}, \cdot)\times \boldsymbol \nu\|_{L^{\infty}(0,T)}.\notag
\end{align}
Therefore,
\begin{align}
\left\|\boldsymbol{a}(\cdot)-\tilde{\boldsymbol{a}}(\cdot)\right\|_{L^{\infty}(0,T_0)} \leq C_2 \| \boldsymbol{H}(\boldsymbol{x}, \cdot)\times\boldsymbol{\nu} - \widetilde{\boldsymbol H}(\boldsymbol{x}, \cdot)\times \boldsymbol \nu\|_{L^{\infty}(0,T)}, \notag
\end{align}
where
\begin{equation}
C_2=2d_0\left\|X_0^{-1}\right\|_\infty C_1. \notag
\end{equation}
The proof is complete.
\end{proof}
\begin{corollary}[Uniqueness]\label{thm}
Let $\boldsymbol{x_1},\,\boldsymbol{x_2},\,\boldsymbol{x_3},\,\boldsymbol{x_4}\in \G$ be four non-coplanar points. The orbit function $\boldsymbol{a}(t)$ for $t\in (0, T_0)$ can be uniquely determined by the measurement data $\nu(\boldsymbol{x}_j) \times \boldsymbol{H}(\boldsymbol{x}_j, t)$ for $j=1,2,3,4$ and $t\in(0, T_0+T_1)$ where $T_1>T_0$ is given by \eqref{T1}.

\end{corollary}

\section{Numerical experiments}\label{numerics}

In this section, we present a couple of numerical experiments to demonstrate the effectiveness and feasibility of the proposed reconstruction method.
All numerical experiments are implemented by MATLAB.
We obtain the synthetic data for the solution of forward problem by the representation:
\[
\boldsymbol H(\boldsymbol x, t)=
\boldsymbol \He(t-|\boldsymbol{x}-\boldsymbol{a}(0)|/c)\frac{\boldsymbol{f}(g^{-1}(t))}{4\pi g^\prime(g^{-1}(t)) |\boldsymbol{x}-\boldsymbol{a}(g^{-1}(t))|},
\]
where
\begin{equation*}
g(t)=t+\frac{|\boldsymbol x -\boldsymbol a(t)|}{c},\quad g'(t)=1+\frac{\boldsymbol a(t)\cdot(\boldsymbol x -\boldsymbol a(t))}{c|\boldsymbol x -\boldsymbol a(t)|}.
\end{equation*}
Given $\boldsymbol f$ and $\boldsymbol a$,
the synthetic data can be obtained once we have $g^{-1}$. Since the function $g^{-1}$ may not be explicit,
we adopt the method of dichotomy to compute $g^{-1}$ (see Algorithm \ref{A1}, \eqref{alg:inverse_function_solving_algorithm}), based on the monotonicity and continuity of $g$. To test the stability of our method with respect to noise, we pollute the synthetic data by
\[
\tilde{\boldsymbol H}(\boldsymbol x, t) = \boldsymbol H(\boldsymbol x, t)(1+2\varepsilon\,{\rm rand} - \varepsilon),
\]
where $\varepsilon>0$ denotes the noise level and {\rm rand} is an independently and uniformly distributed random number in $[0, 1]$. We employ the relative error \text{Err} to measure the reconstructed results:
\[
\text{Err}:=\frac{\|\tilde{\boldsymbol{a}}(\cdot)-{\boldsymbol{a}(\cdot)\|_{L^\infty(0,T_0)}}}{\|{\boldsymbol{a}(\cdot)\|_{L^\infty(0,T_0)}}},
\]
where \( \boldsymbol{a} \) represents the original orbit, and \( \tilde{\boldsymbol{a}} \) represents the reconstructed orbit.

\begin{algorithm}[H]\label{A1}
    \DontPrintSemicolon
    \SetAlgoLined
    \label{alg:inverse_function_solving_algorithm}
    \KwIn{Function $g$, target value $r$, precision $\varepsilon_{inv} = 10^{-7}$}
    \KwOut{Inverse function $g^{-1}(r)$}
    $t \gets 0$ \;
    $\Delta t \gets 1$\;
    \While{$\Delta t > \varepsilon_{inv}$}{
        \uIf{$t > T$ or $ g(t) > r$}{
            $t \gets t - \Delta t$\;

            $\Delta t \gets \Delta t/10$\;
        }
        \Else{
            $t \gets t + \Delta t$\;
        }
    }
    $t \gets t + \Delta t*10$\;
    \Return{$t$}\;
    \caption{Inverse function of $g(t)$}
\end{algorithm}

\subsection{Numerical examples with a high wave speed} In this section the wave speed $c$ is chosen to be the speed of light in a vacuum, which is $c=3\times {10}^8$.
 We present three numerical experiments 1-3 to recover a straight line, a hear-shaped closed curve and a spiral curve in three dimensions. The four boundary observation points are located at
\begin{align}
\boldsymbol{x}_1=\frac{20000}{\sqrt{3}}\begin{bmatrix}
    1 \\
    1 \\
    1
\end{bmatrix},\,
\boldsymbol{x}_2=\frac{20000}{\sqrt{3}}\begin{bmatrix}
    -1 \\
    -1 \\
    1
\end{bmatrix},\, \notag \\
\boldsymbol{x}_3=\frac{20000}{\sqrt{3}}\begin{bmatrix}
    1 \\
    -1 \\
    -1
\end{bmatrix},\,
\boldsymbol{x}_4=\frac{20000}{\sqrt{3}}\begin{bmatrix}
    -1 \\
    1 \\
    -1
\end{bmatrix}. \notag
\end{align}
All of them lie on the sphere $\Omega$ centered at the origin with radius 20000. In our examples, the orbits \( \boldsymbol a \) are contained in the region \( \Omega \). The vector-valued temporal function $\boldsymbol f $ is given by $$\boldsymbol f(t)= \begin{bmatrix}
    1 \\
    15 + 10\sin(100 t) \\
    -1-t^2
\end{bmatrix}. $$  The time interval length $T_0$ of the orbit is $ T_0=2\pi \times 10^{-2}$ and the additional
measurement time $T_1 $ is $T_1= \frac{40000}{c}$.

We compute $T(\boldsymbol x_j), j = 1, 2, 3, 4,$ defined by (\ref{Equality: Definition of t_x }) by choosing components of the magnetic field as follows:
\begin{align*}
    T(\boldsymbol x_1)&=\varepsilon_{ini}\inf\{k\in \mathbb{N}: H_1(\boldsymbol{x}_1,k\varepsilon_{ini})\neq 0\}, T(\boldsymbol x_2)=\varepsilon_{ini}\inf\{k\in \mathbb{N}: H_2(\boldsymbol{x}_2,k\varepsilon_{ini})\neq 0\},\\
   T(\boldsymbol x_3)&=\varepsilon_{ini}\inf\{k\in \mathbb{N}: H_3(\boldsymbol{x}_3,k\varepsilon_{ini})\neq 0\}, T(\boldsymbol x_4)=\varepsilon_{ini}\inf\{k\in \mathbb{N}: H_3(\boldsymbol{x}_4,k\varepsilon_{ini})\neq 0\}.
\end{align*}
where $\varepsilon_{ini} = 10^{-7}$.
Then the following initial value problem is solved by using the fourth-order Runge-Kutta method with a step size of $10^{-5}$: 
\begin{align*}
\begin{cases}
        v'(t;\boldsymbol{x}_j) = \frac{c}{4\pi v(t;\boldsymbol{x}) H_i(\boldsymbol{x}, t + v(t;\boldsymbol{x})/c)} f_i(t) - c, & t \in (0,T_0],\\[5pt]
        v(0;\boldsymbol{x}_j) = c T(\boldsymbol x_j),
 \end{cases}
\end{align*}
where $i=1, 2, 3,$ and $j=1, 2, 3, 4.$
We use the measurement data $H_1(\boldsymbol{x}_1,t)$, $H_2(\boldsymbol{x}_2,t)$ to compute $v(t;\boldsymbol{x}_1)$ and $v(t;\boldsymbol{x}_2)$, respectively, and use
the data $H_3(\boldsymbol{x}_j,t), j=3, 4,$ to compute $v(t;\boldsymbol{x}_3)$ and $v(t;\boldsymbol{x}_4)$, where $H_j(\boldsymbol x, t)$ denotes the $j$-th component of $\boldsymbol{H}(\boldsymbol x, t)\times\boldsymbol\nu(\boldsymbol x)$.

Having computed the distance functions $|\boldsymbol{x}_j-a(t)|$ at the four points, one can determine the moving orbit by directly solving the system of equations (\ref{Equation: Linear system of equations satisfied by v  }).

\textbf{Example 1.} Suppose that the point source moves along the straight line
$$\boldsymbol{a}(t)=\begin{bmatrix}
    1000t \\
    0 \\
    0
\end{bmatrix},\quad t\in (0,T_0),$$
with the constant velocity $\boldsymbol{a}'(t)=1000$.
We solve the inverse problem by using the data measured at six different noise levels $\varepsilon = 0, 10^{-4}, 2\times 10^{-4},  3\times 10^{-4},  4\times 10^{-4},  5\times 10^{-3}.$
Figure \ref{figure:The reconstructed line orbits under varying levels of noise} shows the numerical results, where the reconstructed moving orbit (blue line) is plotted
against the exact orbit (red line) for each noise level. It is clear to see that better reconstructions can be obtained at smaller noise levels. We are particularly interested in the relation between the relative error and the noise level.   
% Overall the orbit is reconstructed with high accuracy.
Table \ref{tab:The variation of relative error for the line orbit} presents the relative errors of the reconstructions against different noise levels.
It can be observed that the relative errors exhibit a linear increase trend with respect to the noise level $\varepsilon$ from Figure \ref{figure:The relative errors for the reconstruction of a straight moving orbit}, which can be explained from the Lipschitz stability estimate shown in \ref{thm}. However, it remains open how does the stability constant on the wave speed. 
As seen from Figure \ref{figure:The reconstructed line orbits under varying levels of noise}, our method does not give a satisfactory reconstruction from the polluted data at the noisy level $\epsilon=0.5\%$ for electromagnetic waves. It will be shown later that good reconstructions from high-noise level data can still be achieved if the wave speed is lower.

\begin{figure}
\centering
\includegraphics[width=1\textwidth]{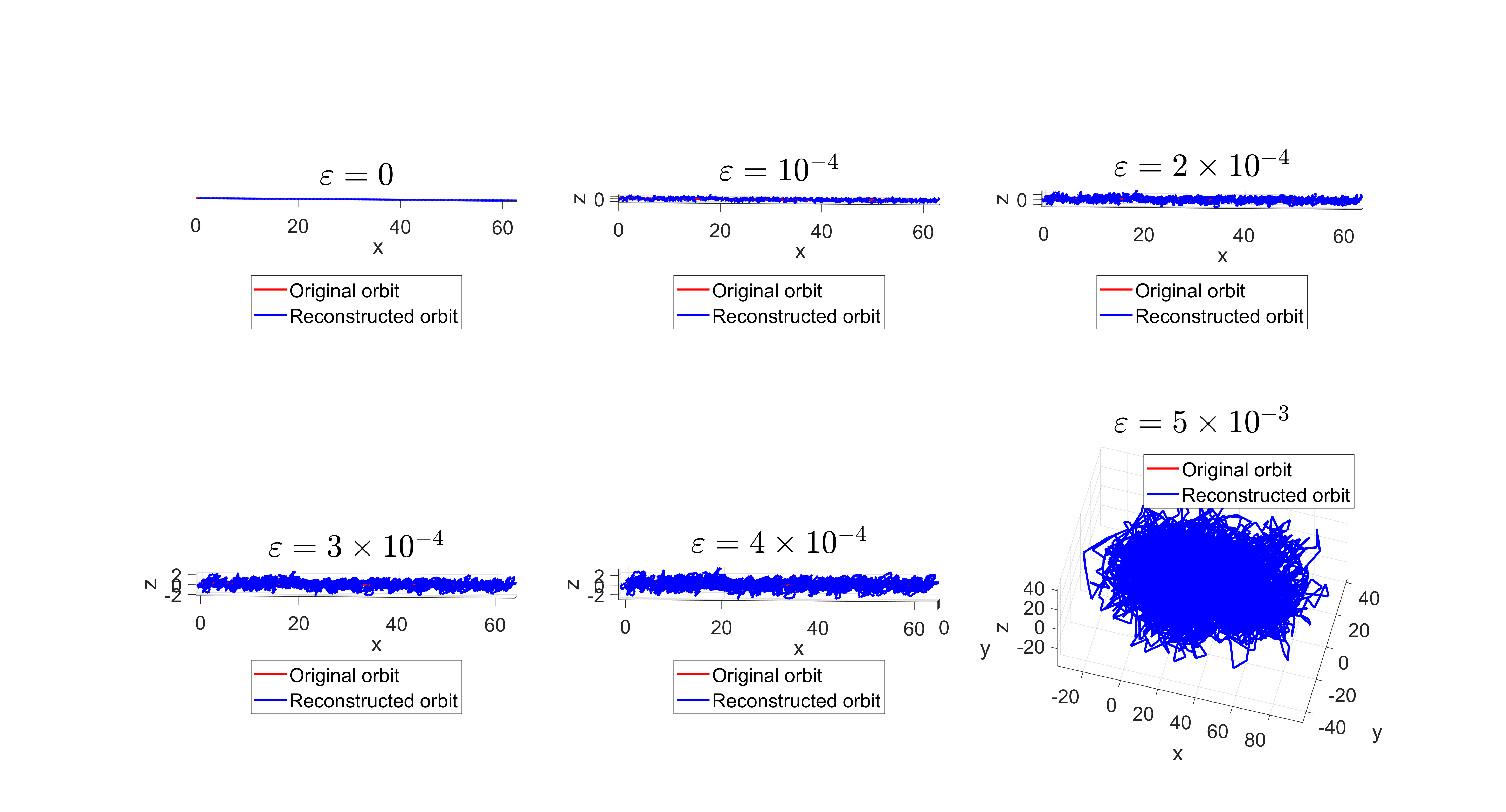}
\caption{Reconstructed orbit (blue line) against exact orbit (red line)  at different noise levels.}
\label{figure:The reconstructed line orbits under varying levels of noise}
\end{figure}

\begin{figure}
\centering
\includegraphics[width=0.8\textwidth]{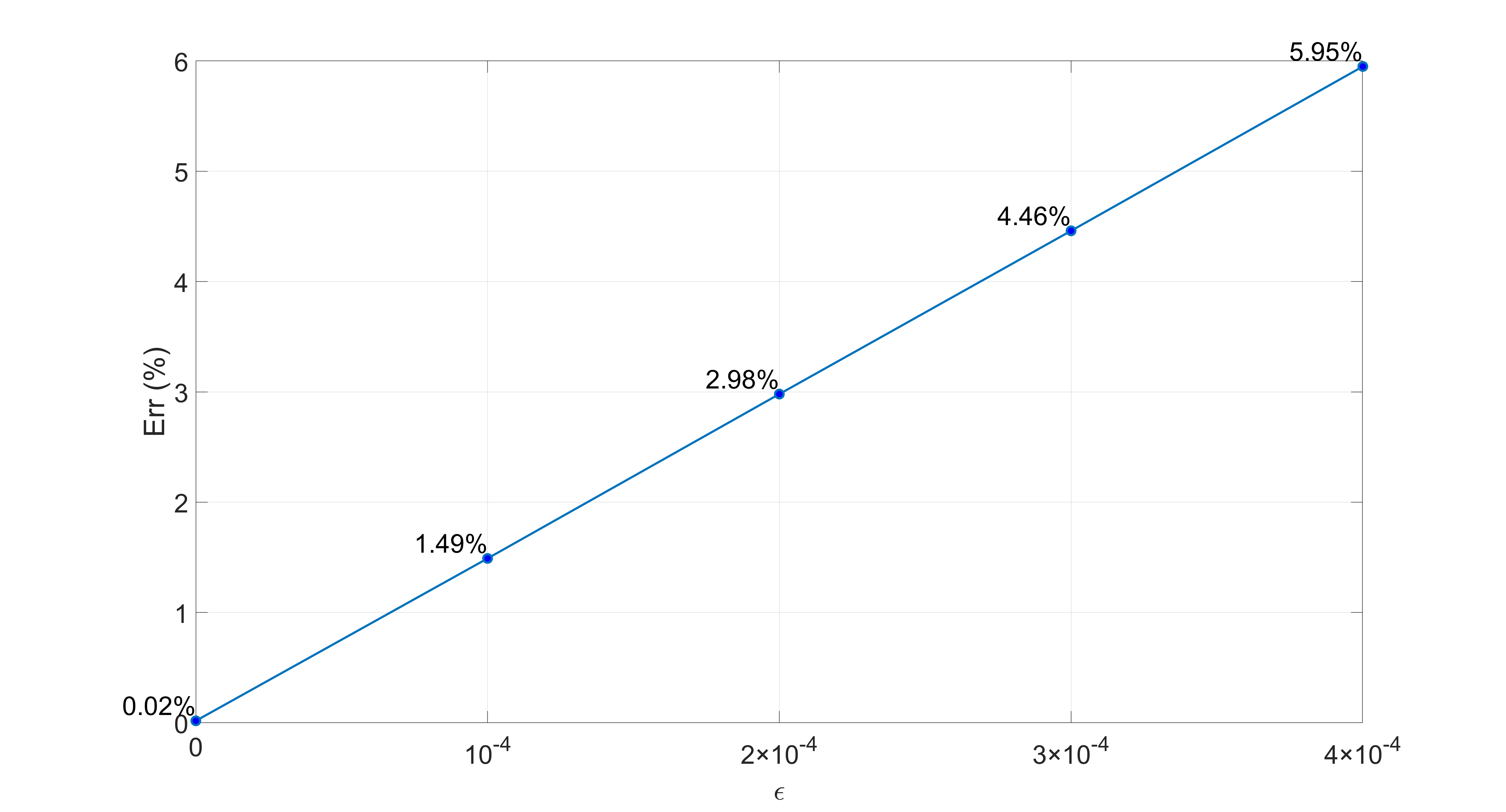}
\caption{Relative errors for the reconstruction of a straightly moving orbit at different noise levels. The $x$-axis represents the noise level $\varepsilon$, while the $y$-axis indicates the relative error Err.}
\label{figure:The relative errors for the reconstruction of a straight moving orbit}
\end{figure}

\begin{table}[htbp]
    \centering
    \caption{Relative errors for the reconstruction of a straightly moving orbit at different noise levels.}
    \label{tab:The variation of relative error for the line orbit}
    \begin{tabular}{ccccccc}
        \hline
        $\varepsilon$  & 0& $10^{-4}$ & $2\times 10^{-4}$ & $3\times 10^{-4}$ & $4\times 10^{-4}$ & $5\times 10^{-3}$ \\
        \hline
        Err &$1.78 \times 10^{-4}$ & 1.49\%& 2.98\%& 4.46\%& 5.95\%&74.53\%\\
        \hline
    \end{tabular}
\end{table}

 \textbf{Example 2.} Consider a  heart-shaped orbit in the $xy$-plane given by
$$\boldsymbol{a}(t)=\begin{bmatrix}
    50 \left(1 - \sin(100t)\right) \cos(100t) \\
    50 \left(1 - \sin(100t)\right) \sin(100t) \\
    0
\end{bmatrix}, t\in (0,T_0).$$
Figure \ref{figure:The reconstructed heart-shaped orbits under varying levels of noise} displays the reconstructed and original orbits under different levels of noise.
Again we are able to reconstruct the orbit with high accuracy, and the relative error is shown in Table \ref{tab:The variation of relative error for the heart-shaped orbit}. As illustrated by the Lipschitz stability estimate \eqref{stability}, Figure \ref{figure:The variation of relative error for the heart-shaped orbit} shows that the relative errors increase linearly with respect to the noise level $\varepsilon$.

\begin{figure}
\centering
\includegraphics[width=1\textwidth]{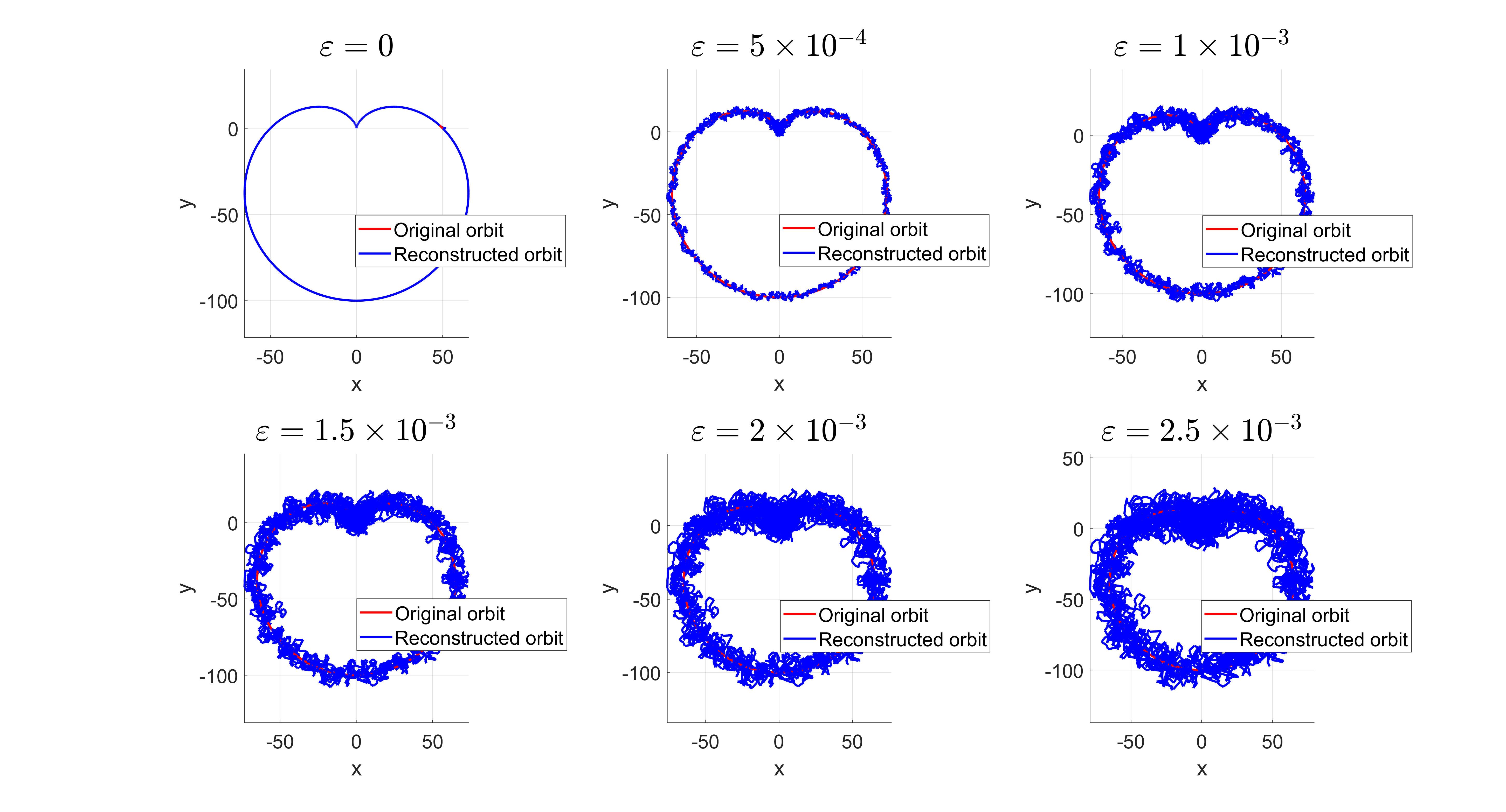}
\caption{Reconstructed (blue line) against exact (red line) solution for a heart-shaped moving orbit. Noise levels are $\varepsilon = 0, 5\times 10^{-4}, 1\times 10^{-3},  1.5\times 10^{-3},  2\times 10^{-3},  2.5\times 10^{-3}.$}
\label{figure:The reconstructed heart-shaped orbits under varying levels of noise}
\end{figure}

\begin{figure}
\centering
\includegraphics[width=0.8\textwidth]{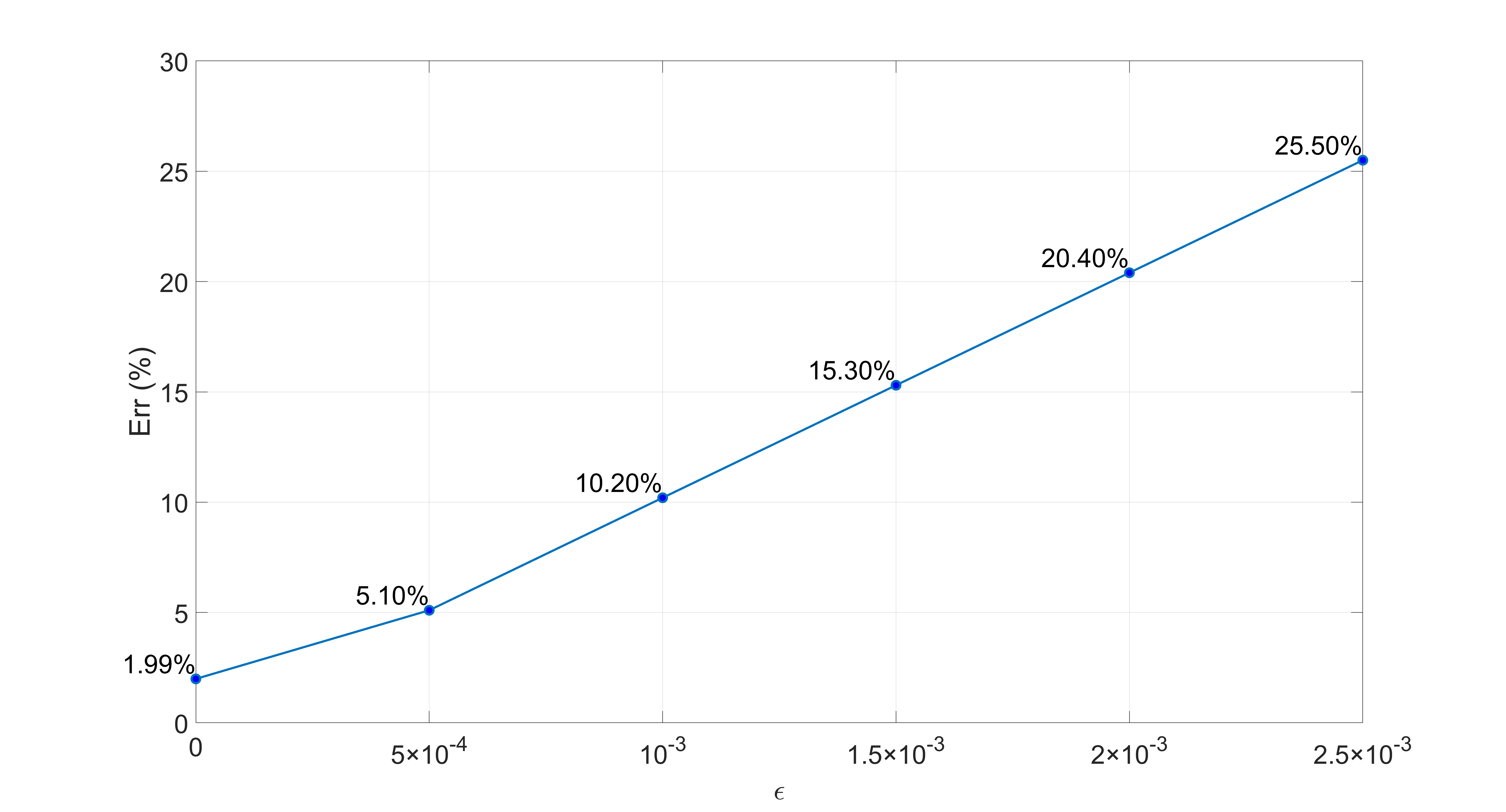}
\caption{Relative errors for the reconstruction of a heart-shaped orbit in the $xy$-plane at different noise levels. The x-axis represents the noise level $\varepsilon$, while the y-axis indicates the relative error Err.}
\label{figure:The variation of relative error for the heart-shaped orbit}
\end{figure}

\begin{table}[htbp]
    \centering
    \caption{The relative errors for the reconstruction of a heart-shaped orbit in the $xy$-plane at different noise levels.}
    \label{tab:The variation of relative error for the heart-shaped orbit}
    \begin{tabular}{ccccccc}
        \hline
        $\varepsilon$  & 0& $5 \times 10^{-4}$ & $1\times 10^{-3}$ & $1.5\times 10^{-3}$ & $2\times 10^{-3}$ & $2.5\times 10^{-3}$ \\
        \hline
        Err & 1.99\% & 5.10\%& 10.20\%& 15.30\%& 20.40\%& 25.50\%\\
        \hline
    \end{tabular}
\end{table}

 \textbf{Example 3.}  Consider a spiral orbit given by
$$\boldsymbol{a}(t)=\begin{bmatrix}
    50 \cos(100t) \\
    50 \sin(100t) \\
    1000t
\end{bmatrix}, t\in (0,T_0).$$

For this more complicated example, the proposed method is also effective and robust to reconstruct the orbit as shown in Figure \ref{figure:The reconstructed spiral orbits under varying levels of noise}, if the noise level is low. The measurement data at the noise level less than $\epsilon=0.1\%$ yield satisfactory reconstructions. Table \ref{tab:The variation of relative error for the spiral orbit} presents the relative error of the reconstructed orbit under different noise levels. Similar to the first two examples, it still maintains a linear increase in the relative error of the reconstructed orbit when the noise level increases as shown as Figure \ref{figure:The relative errors for the reconstruction of a spiral moving orbit}.

\begin{figure}
\centering
\includegraphics[width=1\textwidth]{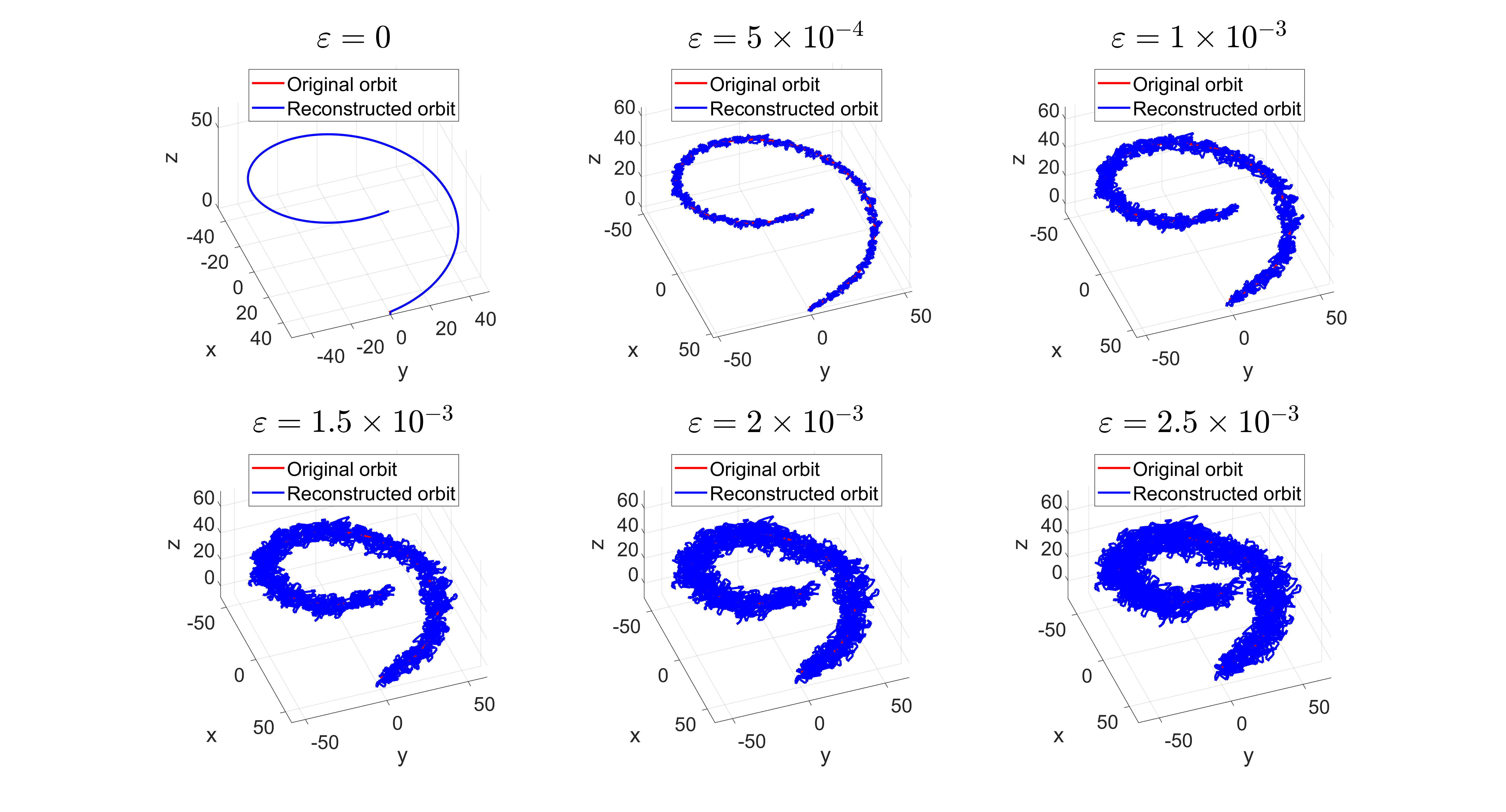}
\caption{Reconstructed (blue line) against exact (red line) solution for a spirally moving orbit. Noise levels are $\varepsilon = 0, 5\times 10^{-4}, 1\times 10^{-3},  1.5\times 10^{-3},  2\times 10^{-3},  2.5\times 10^{-3}.$}
\label{figure:The reconstructed spiral orbits under varying levels of noise}
\end{figure}

\begin{figure}
\centering
\includegraphics[width=0.8\textwidth]{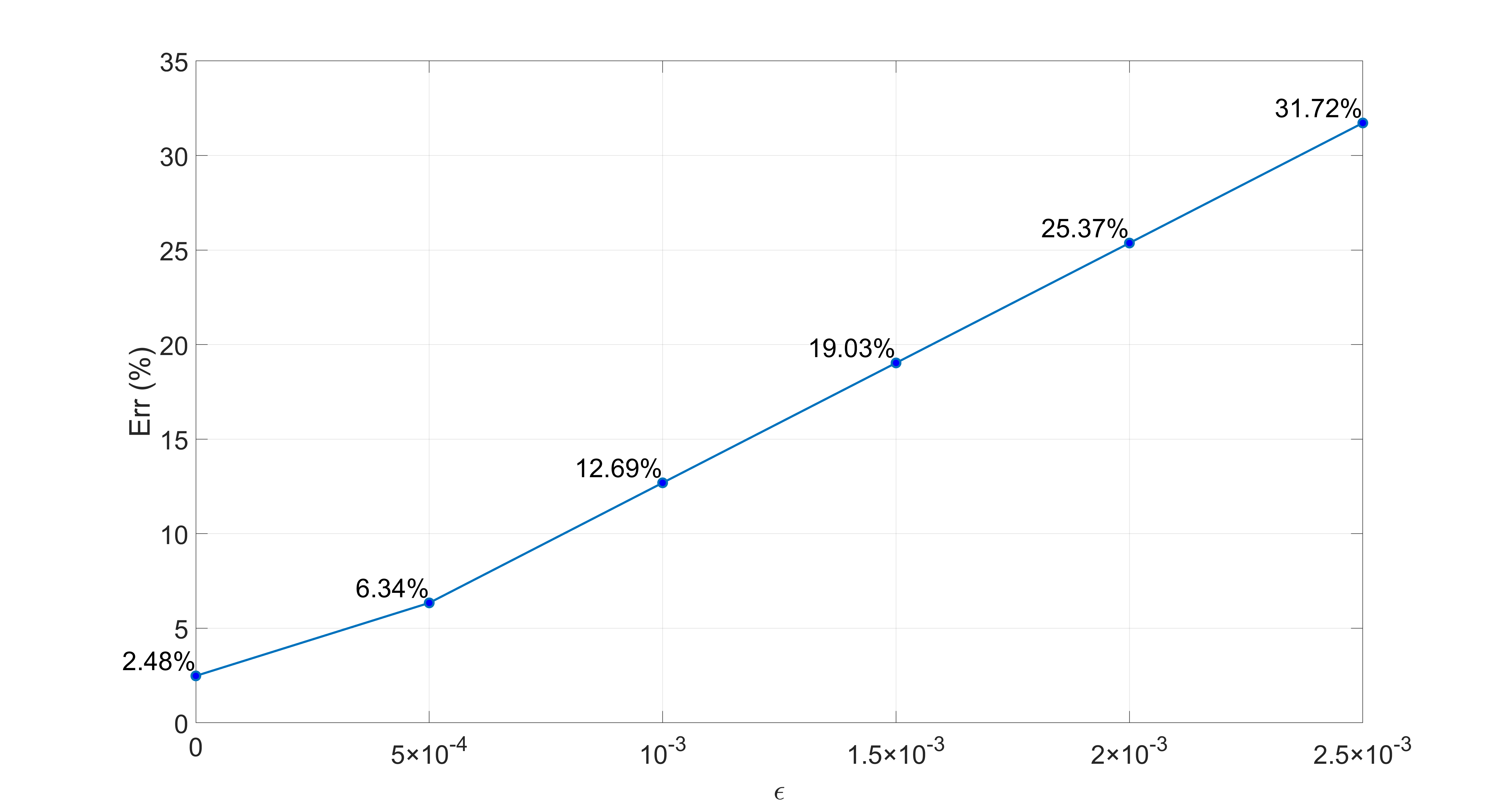}
\caption{Relative errors for the reconstruction of a spirally moving orbit at different noise levels. The $x$-axis represents the noise level $\varepsilon$, while the $y$-axis indicates the relative error Err.}
\label{figure:The relative errors for the reconstruction of a spiral moving orbit}
\end{figure}

\begin{table}[htbp]
    \centering
    \caption{Relative errors for the reconstruction of a spiral orbit at different noise levels.}
    \label{tab:The variation of relative error for the spiral orbit}
    \begin{tabular}{ccccccc}
        \hline
        $\varepsilon$  & 0& $5 \times 10^{-4}$ & $1\times 10^{-3}$ & $1.5\times 10^{-3}$ & $2\times 10^{-3}$ & $2.5\times 10^{-3}$ \\
        \hline
        Err & 2.48\% & 6.34\%& 12.69\%& 19.03\%& 25.37\%& 31.72\%\\
        \hline
    \end{tabular}
\end{table}

\subsection{Numerical examples at a low wave speed.}
We present a numerical experiment with a small wave speed to demonstrate that the wave speed is a crucial factor affecting the error. In this subsection we choose the sound speed $c=340$.
 The time interval length $T_0$ is set as $2\pi \times 10^{-1}$, and the orbit of the moving point source takes the form
$$\boldsymbol{a}(t)=\begin{bmatrix}
    5\cos(10t) \\
    5\sin(10t) \\
    10t
\end{bmatrix}, \quad t\in (0,T_0).$$
Except for the Runge-Kutta method's step size being $1 \times 10^{-4}$, all other settings remain the same as in the previous example. Figure \ref{figure:The reconstructed spiral orbits under varying levels of noise in small wave speed case} presents a comparison between the reconstructed results and the true solution under different noise coefficients. Table \ref{tab:The variation of relative error for the spiral orbit in small wave speed case} shows how the relative error varies with the noise coefficient. It can be observed that, compared with the large wave speed case, the reconstruction with small wave speed can tolerate relatively stronger noise while still maintaining good quality results. We observe from Figure \ref{figure:The reconstructed spiral orbits under varying levels of noise in small wave speed case} that the 12\%-polluted data can give acceptable reconstructions.  It reveals that our method is robust for slowly moving waves. 

\begin{figure}
\centering
\includegraphics[width=1\textwidth]{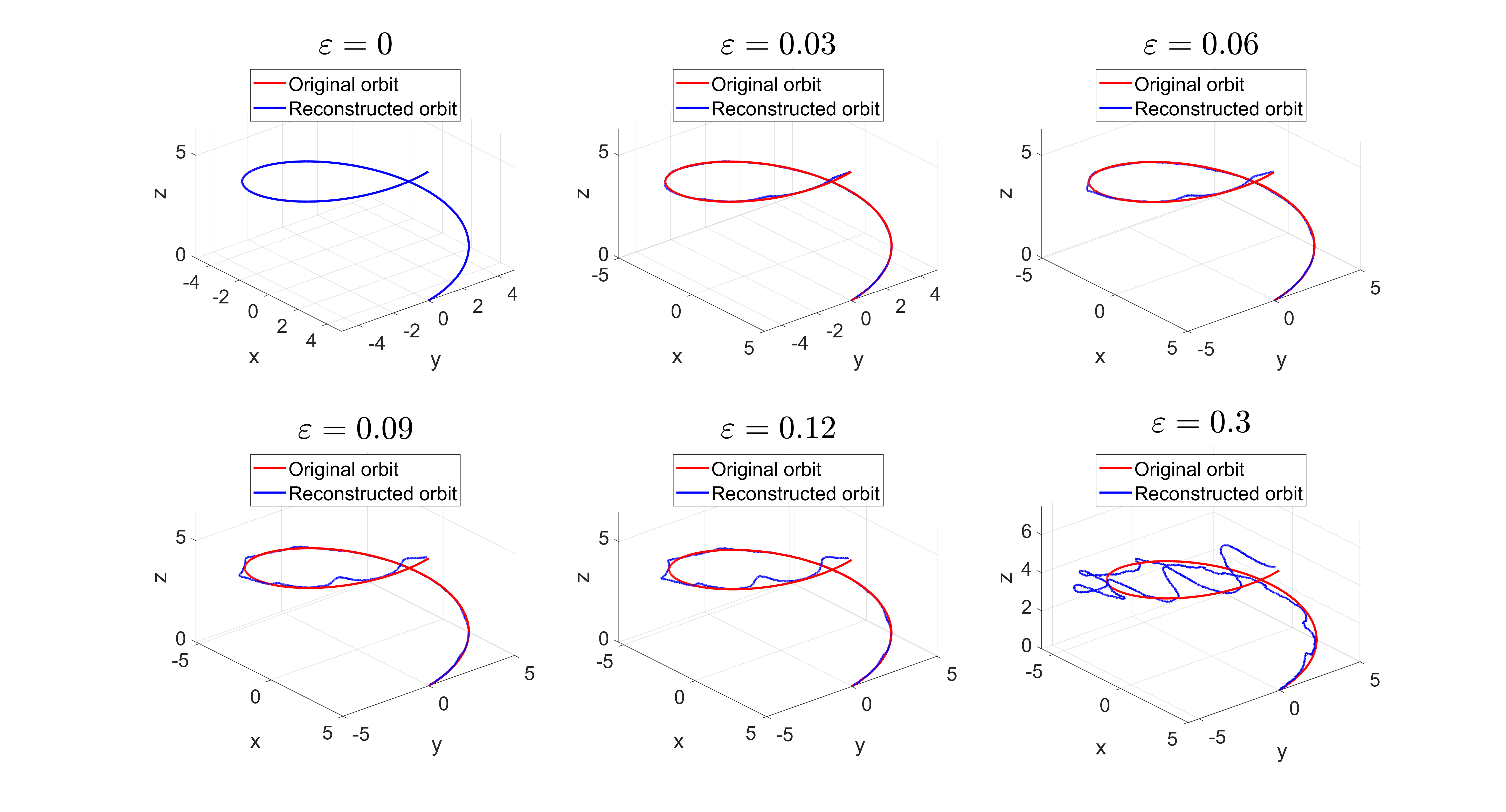}
\caption{Reconstructed (blue line) against exact (red line) solution for a spirally moving orbit (small wave speed case). Noise levels are $\varepsilon = 0, 0.03, 0.06,  0.09, 0.12, 0.3.$}
\label{figure:The reconstructed spiral orbits under varying levels of noise in small wave speed case}
\end{figure}

\begin{figure}
\centering
\includegraphics[width=0.8\textwidth]{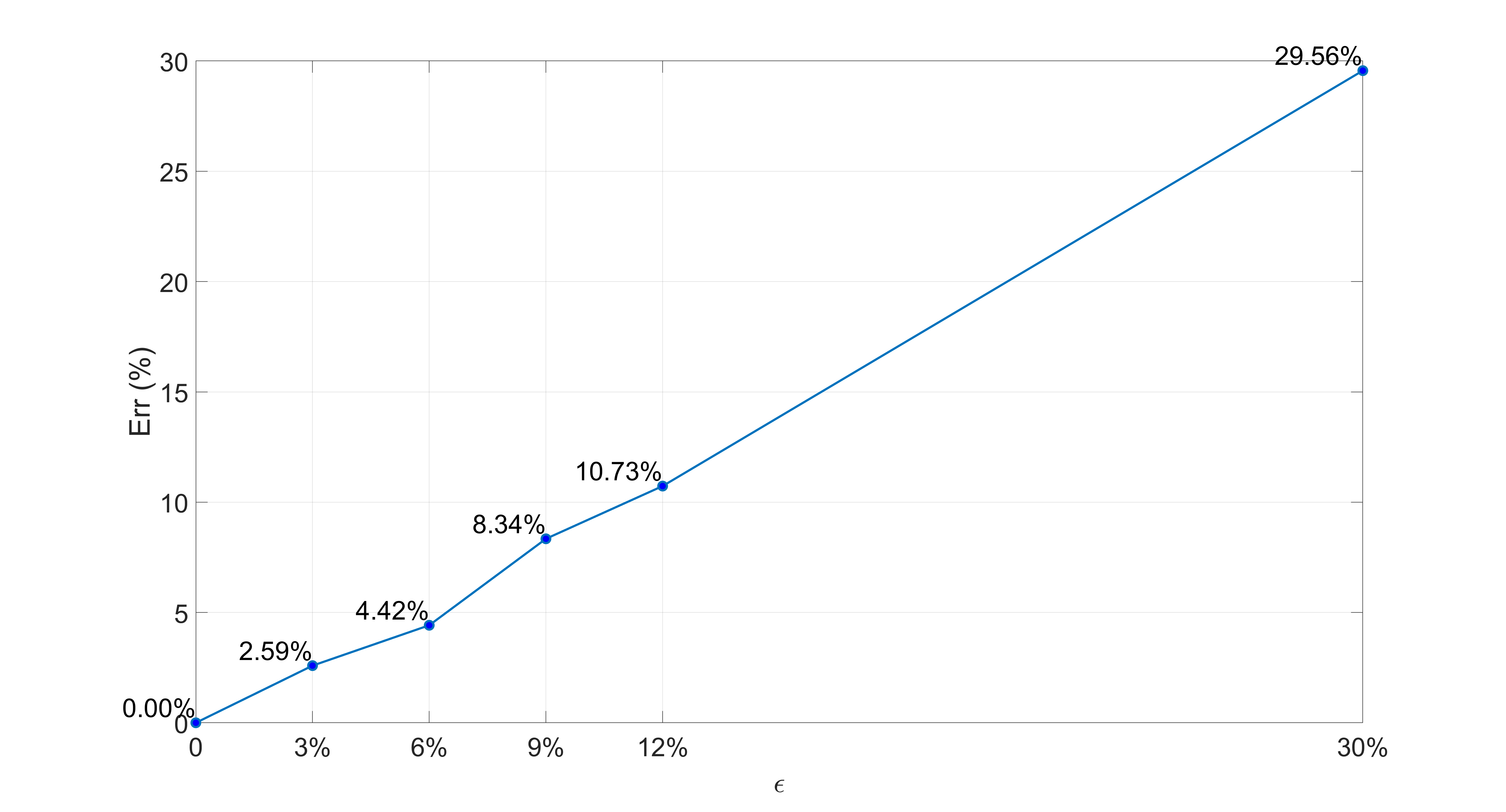}
\caption{Relative errors for the reconstruction of a spirally moving orbit at different noise levels (small wave speed case). The $x$-axis represents the noise level $\varepsilon$, while the $y$-axis indicates the relative error Err.}
\label{figure:The relative errors for the reconstruction of a spiral moving orbit in small wave speed case}
\end{figure}

\begin{table}[htbp]
    \centering
    \caption{Relative errors for spiral orbit reconstruction under different noise levels (small wave speed case)}
    \label{tab:The variation of relative error for the spiral orbit in small wave speed case}
    \begin{tabular}{ccccccc}
        \hline
        $\varepsilon$  & 0& $3\%$ & 6\% & 9\% & 12\% & 30\% \\
        \hline
        Err & $1.11 \times 10^{-6}$ & 2.59\%& 4.42\%& 8.34\%& 10.73\%& 29.56\%\\
        \hline
    \end{tabular}
\end{table}

\textbf{Further discussions on the relative errors.} In the noise-free case, the reconstructed orbit and the true orbit are almost identical, while the emergence of noise introduces a finite deviation between the reconstructed and original ones. As the noise increases, the algorithm still demonstrates certain noise-resistant capabilities. However, when the noise level reaches a certain threshold, the algorithm will fail. The reconstructed orbit will exhibit severe fluctuations, only reflecting the original orbit's trend of change, as demonstrated in Figure \ref{figure:The reconstructed spiral orbits under varying levels of noise} when the noise level is $2.5\times 10^{-3}$. Increasing the solution accuracy of the ordinary differential equations can improve the situation, although some deviation is inevitable.
% As the noise increases, the algorithm still demonstrates certain noise-resistant capabilities, meaning that the reconstructed orbit can reflect the shape of the original orbit. In additional experiments not shown here, we found that higher levels of noise can lead to excessive jaggedness in the reconstructed orbit.
% Increasing the solution accuracy of the ordinary differential equations can significantly improve the situation, allowing the reconstructed orbit to reflect the shape of the original orbit, although some deviation is inevitable.
Furthermore, our numerical experiment with small wave speed reveal that the wave speed is a critical factor affecting the relative error. The algorithm demonstrates significantly greater tolerance to noise under small wave speed conditions compared to large wave speed conditions. 
% This observation is consistent with our stability analysis conclusions.

\section{Conclusion}
In this paper, we develop a numerical strategy to reconstruct the orbit of a moving point source in electromagnetics.
Using the dynamical magnetic field measured at a boundary point generated by a moving point source, we derive an ordinary differential equation (ODE) to describe the distance function between the boundary point and the moving source. By solving the ODE at several receivers we obtain the distance functions with each receiver. As a result, the orbit can be determined by choosing
four non-coplanar boundary points, and numerically we compute the location of the source by solving a linear system of equations.
Based on the reconstruction method, we also establish a Lipschitz stability estimate for the inverse problem. Numerical results show that the developed method
is effective and stable. The proposed approach requires the profile of the source to be a Dirac function. Hence, a possible continuation of this work is to consider the moving source $\boldsymbol F(\boldsymbol x, t) = p(\boldsymbol x - \boldsymbol a(t))\boldsymbol f(t)$, where $p$ is a general function with compact support. The explicit dependence of the stability coefficient on the wave speed deserves to be further investigated. 
More challenging problems include inverse moving point source problems in inhomogeneous media, with several moving orbits and inverse moving obstacle problems through using active emitters.  We  hope to report the progress on these problems
elsewhere in the future.

%\appendix
%
%\section{The free resolvent in $\mathbb R^2$}\label{2}


\begin{thebibliography}{99}

\bibitem{triki_ip}
H. Al Jebawy, A. Elbadia, and F. Triki, Inverse moving point source problem for the wave equation, Inverse Problems, 38 (2022), 125003.

\bibitem{garnier_1}
H. Ammari, E. Bretin, J. Garnier, and A. Wahab, Noise source localization in an attenuating medium, SIAM J. Appl. Math., 72 (2012), 317--336.

\bibitem{bao}
G. Bao, G. Hu, Y. Kian, and T. Yin, Inverse source problems in elastodynamics, Inverse Problems 34 (2018), 045009.

\bibitem{borcea}
L. Borcea, J. Garnier, and K. Solna, Wave propagation and imaging in moving random media,
Multiscale Model. Simul., 17 (2019), 31--67.

\bibitem{borcea_3}
L. Borcea, T. Callaghan, and G. Papanicolaou, Synthetic aperture radar imaging with motion
estimation and autofocus, Inverse Problems, 28 (2012), 045006.

\bibitem{borcea_1}
L. Borcea, J. Garnier, G. Papanicolaou, K. Solna, and C. Tsogka, 
Resolution analysis of passive synthetic aperture imaging of fast moving objects,
SIAM J. Imaging Sci., 10 (2017), 665--710.

\bibitem{garnier}
G. Garnier and M. Fink, Super-resolution in time-reversal focusing on a moving source, Wave
Motion, 53 (2015) 80--93.

\bibitem{Hu}
H. Guo and G. Hu, Inverse wave-number-dependent source problems for the Helmholtz equation,
SIAM J. Numer. Anal. 62 (2024), 1372--1393.

\bibitem{Hu_1}
H. Guo, G. Hu, and G. Ma, Imaging a moving point source from multifrequency data measured at one and sparse observation directions (Part I): Far-field case
SIAM J. Imaging Sci., 16 (2023), 1535--1571.

\bibitem{hu}
G. Hu, P. Li, X. Liu, and Y. Zhao, Inverse source problems in electrodynamics, Inverse Problems and Imaging, 12 (2018), 1411--1428.

\bibitem{hu1}
G. Hu, Y. Kian, P. Li, and Y. Zhao, Inverse moving source problems in electrodynamics, Inverse Problems, 35 (2019), 075001.

\bibitem{hu2}
G. Hu, Y. Kian and Y. Zhao, Uniqueness to some inverse source problems for the wave equation in unbounded domains, Special Issue, Acta Mathematicae Applicatae Sinica, English Series, 36 (2020) 134--150.

\bibitem{isakov}
V. Isakov, Inverse Source Problems,
Math. Surveys Monogr., American Mathematical Society, Providence, RI, 1990.


\bibitem{alge}
E. Nakaguchi, H. Inui, and K. Ohnaka, An algebraic reconstruction of a moving point source for
a scalar wave equation Inverse Problems 28 (2012), 065018.

\bibitem{triki_siap}
S. Wang, M. Karamehmedovi$\acute\rm c$, and F. Triki, Localization of moving sources: uniqueness, stability, and Bayesian inference, 
SIAM J. Appl. Math. 83 (2023), 1049--1073.

\bibitem{liu}
X. Wang, Y. Guo, J. Li, and H. Liu, Mathematical design of a novel input/instruction device using
a moving acoustic emitter Inverse Problems 33 (2017), 105009.

\bibitem{yama}
M. Yamamoto, On an inverse problem of determining source terms in Maxwell’s equations with a single measurement, Inverse Probl. Tomograph. Image Process, New York, Plenum Press, 15 (1998), 241--256







\end{thebibliography}
\end{document}